\documentclass[a4paper, 12pt]{article}

\usepackage{authblk}
\usepackage{hyperref}

\title{Borel-type presentation of the torus-equivariant quantum $K$-ring of flag manifolds of type $C$% Key words
%MSC classification
}
\author[1]{Takafumi Kouno\footnote{e-mail: \url{t.kouno@aoni.waseda.jp}}}
\affil[1]{
Department of Pure and Applied Mathematics,
Faculty of Science and Engineering, 
Waseda University
 3-4-1 Okubo, Shinjuku-ku, Tokyo 169-8555, Japan
}
\author[2]{Satoshi Naito\footnote{e-mail: \url{naito@math.titech.ac.jp}}}
\affil[2]{
Department of Mathematics, 
Institute of Science Tokyo, 
2-12-1 Oh-okayama, Meguro-ku, Tokyo 152-8551, Japan
}
\date{}

\usepackage[margin=25truemm]{geometry}
\usepackage{enumitem}

\usepackage{amsmath, amssymb}
\usepackage{mathtools}
\mathtoolsset{showonlyrefs}
\usepackage{bm}
\numberwithin{equation}{section}

\allowdisplaybreaks[2]

\usepackage{amsthm}
\theoremstyle{plain}
\newtheorem{thm}{Theorem}[section]
\newtheorem{lem}[thm]{Lemma}
\newtheorem{prop}[thm]{Proposition}
\newtheorem{cor}[thm]{Corollary}
\newtheorem{ithm}{Theorem}

\theoremstyle{definition}
\newtheorem{defn}[thm]{Definition}

\theoremstyle{remark}
\newtheorem{rem}[thm]{Remark}
\newtheorem{exm}[thm]{Example}

\newcommand{\BC}{\mathbb{C}}
\newcommand{\BH}{\mathbb{H}}
\newcommand{\BP}{\mathbb{P}}
\newcommand{\BZ}{\mathbb{Z}}

\newcommand{\BI}{\mathbf{I}}

\newcommand{\CA}{\mathcal{A}}
\newcommand{\CE}{\mathcal{E}}
\newcommand{\CF}{\mathcal{F}}

\newcommand{\CI}{\mathcal{I}}
\newcommand{\CJ}{\mathcal{J}}
\newcommand{\CL}{\mathcal{L}}
\newcommand{\CO}{\mathcal{O}}
\newcommand{\CS}{\mathcal{S}}
\newcommand{\CZ}{\mathcal{Z}}
\newcommand{\FE}{\mathfrak{E}}
\newcommand{\FF}{\mathfrak{F}}
\newcommand{\FH}{\mathfrak{H}}
\newcommand{\FP}{\mathfrak{P}}
\newcommand{\FQ}{\mathfrak{Q}}
\newcommand{\FZ}{\mathfrak{Z}}
\newcommand{\st}{\mathsf{t}}
\newcommand{\sq}{\mathsf{q}}
\newcommand{\SD}{\mathsf{D}}
\newcommand{\ST}{\mathsf{T}}
\newcommand{\SX}{\mathsf{X}}

\newcommand{\e}{\mathbf{e}}
\newcommand{\ve}{\varepsilon}
\newcommand{\vp}{\varphi}
\newcommand{\vpi}{\varpi}

\newcommand{\af}{\mathrm{af}}
\newcommand{\poly}{\mathrm{poly}}
\newcommand{\rat}{\mathrm{rat}}

\newcommand{\bQG}{\mathbf{Q}_{G}}
\newcommand{\Sp}{\mathrm{Sp}_{2n}(\BC)}
\newcommand{\KTCQG}{K_{T \times \BC^{\ast}}(\bQG)}
\newcommand{\KTQG}{K_{T}(\bQG)}
\newcommand{\sinf}{\frac{\infty}{2}}

\newcommand{\pair}[2]{\langle #1, #2 \rangle}

\newcommand{\bra}[1]{[\![ #1 ]\!]}
\newcommand{\pra}[1]{(\!( #1 )\!)}

\DeclareMathOperator{\End}{End}
\DeclareMathOperator{\QBG}{QBG}
\DeclareMathOperator{\ed}{end}
\DeclareMathOperator{\down}{down}

\DeclareMathOperator{\sign}{sign}

\newenvironment{enu}{%
\begin{enumerate}[label=\textup{(\arabic*)}]}{%
\end{enumerate}}

\newcommand{\map}[8][\rightarrow]{%
\setlength{\arraycolsep}{2pt}
\begin{array}{rcccl} %
#2 : & #3 & #1 & #4 & #5 \\ %
 & #6 & \mapsto & #7 & #8 %
\end{array}}

\begin{document}
\maketitle

\begin{abstract}
We give a Borel-type presentation of the torus-equivariant (small) quantum $K$-ring of flag manifolds of type $C$. 
\end{abstract}

\section{Introduction}
Let $G$ be a connected, simply-connected, simple algebraic group over $\mathbb{C}$, $B$ its Borel subgroup, $T \subset B$ a maximal torus, and $G/B$ the flag manifold. 
The $T$-equivariant (small) quantum $K$-ring $QK_{T}(G/B)$ of the flag manifold $G/B$, defined by Givental \cite{Givental} and Lee \cite{Lee}, is, as a module, 
\begin{equation}
QK_{T}(G/B) := K_{T}(G/B) \otimes_{R(T)} R(T)\bra{Q_{1}, \ldots, Q_{n}}, 
\end{equation}
where $K_{T}(G/B)$ is the (ordinary) $T$-equivariant $K$-ring of $G/B$, $R(T)$ is the representation ring of $T$, and $Q_{k}$, $k = 1, \ldots, n$, with $n = \mathrm{rank} \, G$, are the (Novikov) variables. 
The quantum $K$-ring $QK_{T}(G/B)$ is also equipped with the quantum product $\star$, defined in terms of the 2-point and 3-point (genus 0, equivariant) $K$-theoretic Gromov-Witten invariants, and becomes an associative, commutative algebra (for details, see \cite{Givental} and \cite{Lee}). 

The quantum $K$-ring $QK_{T}(G/B)$ has a good basis as an $R(T)\bra{Q_{1}, \ldots, Q_{n}}$-module, called the \emph{Schubert basis}. 
Let $W$ be the Weyl group of $G$, and $B^{-}$ the opposite Borel subgroup of $G$ such that $B \cap B^{-} = T$. Then, for $w \in W$, the (opposite) \emph{Schubert variety} $X^{w}$ is defined by $X^{w} := \overline{B^{-}wB/B}$, where $\overline{{}\cdot{}}$ denotes the Zariski closure; 
we denote by $\CO^{w}$, $w \in W$, the structure sheaf of $X^{w}$. 
Then the classes $[\CO^{w}] \in K_{T}(G/B)$, $w \in W$, form an $R(T)$-basis of $K_{T}(G/B)$, and hence they form an $R(T)\bra{Q_{1}, \ldots, Q_{n}}$-basis of $QK_{T}(G/B)$, called the \emph{Schubert basis}. 

The main interest in quantum Schubert calculus lies in determining the structure constants of $QK_{T}(G/B)$ with respect to the Schubert basis $[\CO^{w}]$, $w \in W$, that is, the elements $c_{v, w}^{u, \bm{k}} \in R(T)$ for $v, w, u \in W$ and $\bm{k} \in (\BZ_{\ge 0})^{n}$, given by: 
\begin{equation}
[\CO^{v}] \star [\CO^{w}] = \sum_{\substack{u \in W \\ \bm{k} = (k_{1}, \ldots, k_{n}) \in (\BZ_{\ge 0})^{n}}} c_{v, w}^{u, \bm{k}} Q_{1}^{k_{1}} \cdots Q_{n}^{k_{n}} [\CO^{u}]. 
\end{equation}
One approach to this problem is to give a (Borel-type) presentation of the quantum $K$-ring $QK_{T}(G/B)$ as the quotient of a Lauren polynomial ring by an explicit ideal, and then find a ``good'' Laurent polynomial representing each Schubert class under the above presentation. 
In type $A$, Maeno--Naito--Sagaki \cite{MNS} (see also \cite{ACT}) has given a Borel-type presentation of $QK_{T}(Fl_{n+1})$, where $Fl_{n+1}$ denotes the flag manifold $G/B$ of type $A_{n}$, where $G = SL_{n+1}(\BC)$ and $B$ is its subgroup consisting of the upper-triangular matrices; 
cf. the original (conjectural) presentation proposed by Kirillov--Maeno. 
Moreover, Maeno--Naito--Sagaki \cite{MNS2} proved that \emph{quantum double Grothendieck polynomials}, introduced by Lenart--Maeno \cite{LM}, 
represent Schubert classes under the Borel-type presentation of $QK_{T}(Fl_{n+1})$. 

In this paper, as the first (but important) step in the above approach in type $C$, we give a Borel-type presentation of $QK_{T}(G/B)$, where $G/B$ is the flag manifold of type $C$. 
Let $G = \Sp$ be the symplectic group of rank $n$, and 
$\{\ve_{1}, \ldots, \ve_{n}\}$ the standard basis for the weight lattice $P$ of $G$. 
For $\lambda \in P$, we denote by $\e^{\lambda} \in R(T)$ the character of the one-dimensional representation of $B$ whose weight is $\lambda$. 
We introduce certain elements $F_{1}, \ldots, F_{n} \in (R(T)\bra{Q_{1}, \ldots, Q_{n}})[z_{1}^{\pm 1}, \ldots, z_{n}^{\pm 1}]$ (see Definition~\ref{def:Fl} for the explicit form of these elements), and define an ideal $\CI^{Q}$ of $(R(T)\bra{Q_{1}, \ldots, Q_{n}})[z_{1}^{\pm 1}, \ldots, z_{n}^{\pm 1}]$ to be the one generated by 
\begin{equation}
F_{l} - e_{l}(\e^{\ve_{1}}, \ldots, \e^{\ve_{n}}, \e^{-\ve_{n}}, \ldots, \e^{-\ve_{1}}), \quad 1 \le l \le n, 
\end{equation}
where $e_{l}(x_{1}, \ldots, x_{2n})$ is the $l$-th elementary symmetric polynomial in the ($2n$) variables $x_1, \ldots, x_{2n}$ for $1 \le l \le n$. 
It follows from the definition that the elements $F_{1}, \ldots, F_{n}$ satisfy $F_{l}|_{Q_{1} = \cdots = Q_{n} = 0} = e_{l}(z_{1}, \ldots, z_{n}, z_{n}^{-1}, \ldots, z_{1}^{-1})$. 

The following is the main result of this paper. 

\begin{ithm}[={} Theorem~\ref{thm:main}]
There exists an $R(T)\bra{Q}$-algebra isomorphism 
\begin{equation} \label{eq:quantum_Borel_intro}
\map[\xrightarrow{\sim}]{\Psi^{Q}}{(R(T)\bra{Q})[z_{1}^{\pm 1}, \ldots, z_{n}^{\pm 1}]/\CI^{Q}}{QK_{T}(G/B)}{}{z_{j} + \CI^{Q}}{\dfrac{1}{1-Q_{j}} [\CO_{G/B}(-\ve_{j})],}{\quad 1 \le l \le n,} 
\end{equation}
where $R(T)\bra{Q} := R(T)\bra{Q_{1}. \ldots, Q_{n}}$. 
\end{ithm}

If we set $Q_{1} = \cdots = Q_{n} = 0$, then the isomorphism $\Psi^{Q}$ above specialize to the $R(T)$-algebra isomorphism
\begin{equation} \label{eq:Borel_intro}
R(T)[z_{1}^{\pm 1}, \ldots, z_{n}^{\pm 1}]/\CI \simeq K_{T}(G/B), 
\end{equation}
where $\CI$ is the ideal of $R(T)[z_{1}^{\pm 1}, \ldots, z_{n}^{\pm 1}]$ 
generated by 
\begin{equation}
e_{l}(z_{1}, \ldots, z_{n}, z_{n}^{-1}, \ldots, z_{1}^{-1}) - e_{l}(\e^{\ve_{1}}, \ldots, \e^{\ve_{n}}, \e^{-\ve_{n}}, \ldots, \e^{-\ve_{1}}), \quad 1 \le l \le n; 
\end{equation}
this isomorphism is just the classical Borel presentation (see \cite{PR}): 
\begin{equation}
R(T) \otimes_{R(T)^{W}} R(T) \simeq K_{T}(G/B)
\end{equation}
of the flag manifold of type $C_{n}$, 
where $R(T)^{W}$ is the subalgebra of $R(T)$ consisting of the $W$-invariant elements. 

The explicit form of the elements $F_{1}, \ldots, F_{n}$ come from certain identities in the $T$-equivariant $K$-group $K_{T}(\bQG)$ of the semi-infinite flag manifold $\bQG$ associated to $G = \Sp$, which is a (reduced) scheme of infinite type (see Section~\ref{sec:semi-infinite}), through the $R(T)$-module isomorphism $QK_{T}(G/B) \simeq K_{T}(\bQG)$ established in \cite{Kato}; for this isomorphism, see Theorem~\ref{thm:quantum=semi-infinite}. 
More precisely, we know certain identities in $K_{T}(\bQG)$, called the \emph{inverse Chevalley formula}; in types $A, D, E$, this formula is obtained in the general case by \cite{KNOS} and \cite{LNOS}, and in type $C$, it is obtained in some special cases by \cite{KNO}. 
From the inverse Chevalley formula, we deduce ``key'' relations (see \eqref{eq:system}) for the elements $\FF_{l} \in K_{T}(\bQG)$, $0 \le l \le n$, defined in Definition~\ref{def:FFl}. 
These key relations can be regarded as recurrence relations for $\FF_{l}$ with coefficients in $R(T)$ from which the $\FF_{l} \in K_{T}(\bQG)$, $0 \leq l \leq n$, are uniquely determined. Hence, by solving these recurrence relations, 
we obtain an explicit description (see Theorem~\ref{thm:relation_semi-infinite}) 
of $\FF_{l}$, $0 \le l \le n$, as elements of $R(T)$. 
Finally, by sending the $\FF_{l}$, $0 \le l \le n$, to $QK_{T}(G/B)$ through the isomorphism $QK_{T}(G/B) \simeq K_{T}(\bQG)$, we conclude that some explicit relations hold in $QK_{T}(G/B)$ (see Corollary~\ref{cor:relation_QK}). 
Thus we can figure out the explicit form, given in Definition~\ref{def:Fl}, of the generators of the ideal $\CI^{Q}$. 

Finally, in order to complete the proof of Theorem~\ref{thm:main}, 
we use Nakayama-type arguments. Namely, we make use of a lemma (see Lemma~\ref{lem:quotient_isomorphism_lemma}), due to Gu--Mihalcea--Sharpe--Zou \cite{GMSZ}, in commutative ring theory. By this lemma, we can conclude that \eqref{eq:Borel_intro} ``implies'' \eqref{eq:quantum_Borel_intro}.

This paper is organized as follows. 
In Section~\ref{sec:preliminary}, we fix our basic notation for root systems, in particular, for the root system of type $C$. 
Also, we briefly recall the definition of the $T$-equivariant quantum $K$-ring $QK_{T}(G/B)$ of the flag manifold $G/B$. 
In Section~\ref{sec:main}, we state a Borel-type presentation of $QK_{T}(G/B)$ in type $C$, which is the main result of this paper. 
In Section~\ref{sec:semi-infinite}, we recall the definition of the $T$-equivariant $K$-group $K_{T}(\bQG)$ of the semi-infinite flag manifold $\bQG$, and its relationship with $QK_{T}(G/B)$. 
In addition, we review the inverse Chevalley formula for $K_{T}(\bQG)$, which is used in the proof of our Borel-type presentation. 
In Section~\ref{sec:key}, we deduce key relations in $K_{T}(\bQG)$ by using the inverse Chevalley formula. 
In Section~\ref{sec:proof}, we give a proof of our Borel-type presentation of $QK_{T}(G/B)$ in type $C$.

\subsection*{Acknowledgments}
The authors are grateful to Takeshi Ikeda for helpful discussions. 
The second author would like to thank Leonardo Mihalcea for suggesting that we should make use of Lemma~\ref{lem:finitely_generated} to show the finite generation of a quotient ring. 
T.K. was partly supported by JSPS Grants-in-Aid for Scientific Research 22J00874, 22KJ2908, and 24K22842. 
S.N. was partly supported by JSPS Grant-in-Aid for Scientific Research (C) 21K03198.

\section{Basic setting} \label{sec:preliminary}
We fix our basic notation for root systems, in particular, for the root system of type $C$. Also, we briefly recall the definition of the torus-equivariant (small) quantum $K$-ring of flag manifolds. 

\subsection{Notation for root systems}

Let $G$ be a connected, simply-connected, simple algebraic group over $\BC$, $B$ its Borel subgroup, $T \subset B$ a maximal torus, and $I$ the index set for the simple roots of $G$. 
Let $P = \sum_{i \in I} \BZ \varpi_{i}$ be the weight lattice of $G$, with $\varpi_{i}$ the $i$-th fundamental weight, and $Q^{\vee} = \sum_{i \in I} \BZ \alpha_{i}^{\vee}$ the coroot lattice of $G$, with $\alpha_{i}^{\vee}$ the coroot such that $\varpi_{i}(\alpha_{j}^{\vee}) = \delta_{i,j}$ for $i, j \in I$; 
we set $Q^{\vee, +} := \sum_{i \in I} \BZ_{\ge 0} \alpha_{i}^{\vee}$. 
In addition, let $W$ be the Weyl group of $G$, which is generated by the simple reflections $s_{i}$, $i \in I$. 
For $\lambda \in P$, we denote by $\e^{\lambda}$ the character of the one-dimensional representation of $B$ whose weight is $\lambda$. 
Then the representation ring $R(T)$ of $T$ is written as $R(T) = \sum_{\lambda \in P} \BZ \e^{\lambda}$, with product given by $\e^{\lambda} \cdot \e^{\mu} := \e^{\lambda+\mu}$ for $\lambda, \mu \in P$; 
note that $R(T)$ is identified with the group algebra $\BZ[P]$ of $P$, on which  $W$ acts by $w \cdot \e^{\lambda} := \e^{w(\lambda)}$ for $w \in W$ and $\lambda \in P$. 

\subsection{\texorpdfstring{The root system of type $C_{n}$}{The root system of type Cn}} \label{sec:Cn}
We fix the notation for the root system of type $C_{n}$. 
In this subsection, let $G = \Sp$ be the symplectic group of rank $n$. 
Let $\{\ve_{1}, \ldots, \ve_{n}\}$ be the standard basis for $P$. 
Then the fundamental weights of $G$ are $\{\ve_{1}, \ve_{1}+\ve_{2}, \ldots, \ve_{1}+\ve_{2}+\cdots+\ve_{n}\}$, and the set $\Delta^{+}$ of positive roots of $G$ is $\Delta^{+} = \{\ve_{i} \pm \ve_{j} \mid 1 \le i < j \le n\} \cup \{2\ve_{j} \mid 1 \le j \le n\}$, with the simple roots $\{\ve_{i} - \ve_{i+1} \mid 1 \le i \le n-1\} \cup \{2\ve_{n}\}$. 
It follows that $\BZ[P] = \BZ[\e^{\pm \ve_{1}}, \ldots, \e^{\pm \ve_{n}}]$. 

The Weyl group $W \simeq S_{n} \ltimes (\BZ / 2 \BZ)^{n}$ can be regarded as the group of signed permutations on $[1, \overline{1}] := \{1 < 2 < \cdots < n < \overline{n} < \cdots < \overline{2} < \overline{1}\}$. 
In particular, the reflection $s_{(i, j)}$, $1 \le i < j \le n$, corresponding to the root $(i, j) := \ve_{i} - \ve_{j}$ acts on $[1, \overline{1}]$ as follows: 
for $1 \le k \le n$, 
\begin{align}
s_{(i, j)} (k) = \begin{cases} 
j & \text{if $k = i$,} \\ 
i & \text{if $k = j$,} \\ 
k & \text{if $k \not= i, j$;}
\end{cases} \\ 
s_{(i, j)} (\overline{k}) = \begin{cases} 
\overline{j} & \text{if $k = i$,} \\ 
\overline{i} & \text{if $k = j$,} \\ 
\overline{k} & \text{if $k \not= i, j$.}
\end{cases}
\end{align}
The reflection $s_{(i, \overline{j})}$, $1 \le i < j \le n$, corresponding to the root $(i, \overline{j}) := \ve_{i} + \ve_{j}$ acts on $[1, \overline{1}]$ as follows: for $1 \le k \le n$, 
\begin{align}
s_{(i, \overline{j})} (k) = \begin{cases} 
\overline{j} & \text{if $k = i$,} \\ 
\overline{i} & \text{if $k = j$,} \\ 
k & \text{if $k \not= i, j$;}
\end{cases} \\ 
s_{(i, \overline{j})} (\overline{k}) = \begin{cases} 
j & \text{if $k = i$,} \\ 
i & \text{if $k = j$,} \\ 
\overline{k} & \text{if $k \not= i, j$.}
\end{cases}
\end{align}
In addition, the reflection $s_{(i, \overline{i})}$, $1 \le i \le n$, corresponding to the root $(i, \overline{i}) := 2\ve_{i}$ acts on $[1, \overline{1}]$ as follows: for $1 \le k \le n$, 
\begin{align}
s_{(i, \overline{i})} (k) = \begin{cases}
\overline{i} & \text{if $k = i$,} \\ 
k &\text{if $k \not= i$;}
\end{cases} \\ 
s_{(i, \overline{i})} (\overline{k}) = \begin{cases}
i & \text{if $k = i$,} \\ 
\overline{k} &\text{if $k \not= i$.}
\end{cases}
\end{align}
The action of $W$ on $P$ is a natural extension of that on $[1, \overline{1}]$, written as: $w \cdot \ve_{k} = \ve_{w(k)}$ for $w \in W$ and $k = 1, \ldots, n$, where we set $\ve_{\overline{j}} := -\ve_{j}$ for $1 \le j \le n$. 
We sometimes write $w \in W$ as $w = [w(1), w(2), \ldots, w(n)]$ by regarding it as a signed permutation as above; this notation is called the \emph{window notation}. 

\subsection{\texorpdfstring{The quantum $K$-ring of flag manifolds}{The quantum K-ring of flag manifolds}}

Let $G$ be of an arbitrary type, $G/B$ the flag manifold, and $K_{T}(G/B)$ the (ordinary) $T$-equivariant $K$-ring of $G/B$. 
The \emph{$T$-equivariant (small) quantum $K$-ring} $QK_{T}(G/B)$ of the flag manifold $G/B$, defined by Givental \cite{Givental} and Lee \cite{Lee}, is, as a  module, 
\begin{equation}
QK_{T}(G/B) := K_{T}(G/B) \otimes_{R(T)} R(T)\bra{Q_{1}, \ldots, Q_{n}}, 
\end{equation}
where $n$ is the rank of $G$; here, $R(T)\bra{Q_{1}, \ldots, Q_{n}}$ denotes the formal power series ring in the (Novikov) variables $Q_{1}, \ldots, Q_{n}$. 
Under the quantum product $\star$, defined in terms of the 2-point and 3-point (genus 0, equivariant) $K$-theoretic Gromov-Witten invariants, 
$QK_{T}(G/B)$ becomes an associative, commutative algebra (see \cite{Givental} and \cite{Lee} for details). 
For $\xi = \sum_{i \in I} c_{i} \alpha_{i}^{\vee} \in Q^{\vee, +}$, we set $Q^{\xi} := \prod_{i \in I} Q_{i}^{c_{i}}$. 

For $w \in W$, we define the (opposite) \emph{Schubert variety} $X^{w} \subset G/B$ by $X^{w} := \overline{B^{-}wB/B}$, where $B^{-}$ is the opposite Borel subgroup of $G$ such that $B \cap B^{-} = T$, and $\overline{{}\cdot{}}$ denotes the Zariski closure; let $\CO^{w}$ denote the structure sheaf of $X^{w}$. 
Then it is well-known that the classes $[\CO^{w}]$, $w \in W$, of $\CO^{w}$ in $K_{T}(G/B)$ form a basis of $K_{T}(G/B)$ as an $R(T)$-module, which are called 
the (opposite) \emph{Schubert classes}. 
It follows that the (opposite) Schubert classes $[\CO^{w}]$, $w \in W$, also form a basis of $QK_{T}(G/B)$ as an $R(T)\bra{Q_{1}, \ldots, Q_{n}}$-module. 

To each $\lambda \in P$, 
we can associate the line bundle $G \times_{B} \BC_{-\lambda} \twoheadrightarrow G/B$ over $G/B$, denoted by $\CO_{G/B}(\lambda)$, where $\BC_{-\lambda}$ is the one-dimensional representation of $B$ whose weight is $- \lambda$; let 
$[\CO_{G/B}(\lambda)] \in K_{T}(G/B) \subset QK_{T}(G/B)$ denote the class corresponding to the line bundle $\CO_{G/B}(\lambda)$. 

\section{\texorpdfstring{Borel-type presentation of $QK_{T}(G/B)$}{Borel-type presentation of QKT(G/B)}} \label{sec:main}
Let $G = \Sp$ be the symplectic group of rank $n$. We give a presentation of $QK_{T}(G/B)$ as an explicit quotient of a certain Laurent polynomial ring. 

For a commutative ring $R$, we denote by $R\bra{Q} := R\bra{Q_{1}, \ldots, Q_{n}}$ the formal power series ring in the (Novikov) variables $Q_{1}, \ldots, Q_{n}$. 
For $1 \le k \le 2n$, let $e_{k}(x_{1}, \ldots, x_{2n})$ be the $k$-th elementary symmetric polynomial in the ($2n$) variables $x_{1}, \ldots, x_{2n}$, i.e., 
\begin{equation}
e_{k}(x_{1}, \ldots, x_{2n}) := \sum_{1 \le i_{1} < \cdots < i_{k} \le 2n} x_{i_{1}} x_{i_{2}} \cdots x_{i_{k}}. 
\end{equation} 
For $0 \le l \le n$, we define an element $E_{l} \in R(T)$ by 
\begin{equation}
E_{l} := e_{l}(\e^{\ve_{1}}, \ldots, \e^{\ve_{n}}, \e^{-\ve_{n}}, \ldots, \e^{-\ve_{1}}). 
\end{equation}

Also, we consider the Laurent polynomial ring $(R(T)\bra{Q})[z_{1}^{\pm 1}, \ldots, z_{n}^{\pm 1}]$ with coefficients in the formal power series ring $R(T)\bra{Q} = R(T)\bra{Q_1, \ldots, Q_n}$. 

\begin{defn}
Let $I \subset [1, \overline{1}]$. 
\begin{enu}
\item For $1 \le j \le n$, we define $\zeta_{I}^{\poly}(j) \in \BZ[Q] \subset \BZ\bra{Q}$ by 
\begin{equation}
\zeta_{I}^{\poly}(j) := \begin{cases}
1 - Q_{j} & \text{if $j \in I$ and $j+1 \notin I$,} \\ 
1 & \text{otherwise,} 
\end{cases}
\end{equation}
where $n+1$ is understood to be $\overline{n}$. 

\item For $2 \le j \le n$, we define $\zeta_{I}^{\poly}(\overline{j}) \in \BZ\bra{Q}$ by 
\begin{equation}
\zeta_{I}^{\poly}(\overline{j}) := \begin{cases}
1 + \dfrac{Q_{j-1} Q_{j} \cdots Q_{n}}{1 - Q_{j-1}} & \text{if $I$ is of the form $\{ \cdots < j-1 < \overline{j+1} < \cdots \}$,} \\ 
1 -Q_{j-1} & \text{if $\overline{j} \in I$ and $\overline{j-1} \notin I$,} \\ 
1 & \text{otherwise.}
\end{cases}
\end{equation}
\end{enu}
\end{defn}

\begin{exm}
Let $n = 4$ and $I = \{2, 3, \overline{3}, \overline{1}\}$. Then the elements $\zeta_{i}^{\poly}(j)$ for $1 \le j \le \overline{1}$ are as in the following Table~\ref{tab:zeta_list_example}. 
\begin{table}[h]
\centering
\caption{The list of $\zeta_{I}^{\poly}(j)$ for $1 \le j \le \overline{1}$}
\label{tab:zeta_list_example}
\begin{tabular}{|c|cccccccc|} \hline 
$j$ & $1$ & $2$ & $3$ & $4$ & $\rule{0em}{2.5ex}\overline{4}$ & $\overline{3}$ & $\overline{2}$ & $\overline{1}$ \\ \hline
$\zeta_{I}^{\poly}(j)$ & $1$ & $1$ & $1-Q_{3}$ & $1$ & $\rule{0em}{3.7ex}1 + \dfrac{Q_{3}Q_{4}}{1-Q_{3}}$ & $1$ & $1-Q_{1}$ & $1$ \\ \hline
\end{tabular}
\end{table}
\end{exm}

We set $z_{\overline{j}} := z_{j}^{-1}$ for $1 \le j \le n$. 
\begin{defn} \label{def:Fl}
For $0 \le l \le 2n$, we define $F_{l} \in (R(T)\bra{Q})[z_{1}^{\pm 1}, \ldots, z_{n}^{\pm 1}]$ by 
\begin{equation}
F_{l} := \sum_{\substack{I \subset [1, \overline{1}] \\ |I| = l}} \left( \prod_{1 \le j \le \overline{1}} \zeta_{I}^{\poly}(j) \right) \left( \prod_{j \in I} z_{j} \right). 
\end{equation}
\end{defn}

\begin{rem} \label{rem:specialization_Fl}
Under the specialization $Q_{1} = \cdots = Q_{n} = 0$, the $F_{l}$, $0 \le l \le 2n$, become: 
\begin{equation}
F_{l}|_{Q_{1} = \cdots = Q_{n} = 0} = \sum_{\substack{I \subset [1, \overline{1}] \\ |I| = l}} \left( \prod_{j \in I} z_{j} \right) = e_{l}(z_{1}, \ldots, z_{n}, z_{n}^{-1}, \ldots, z_{1}^{-1}).  
\end{equation}

\end{rem}

\begin{defn}
We define an ideal $\CI^{Q}$ of $(R(T)\bra{Q})[z_{1}^{\pm 1}, \ldots, z_{n}^{\pm 1}]$ to be the one generated by the elements $F_{l} - E_{l}$, $1 \le l \le n$. 
\end{defn}

The main result of this paper is the following. 
\begin{thm} \label{thm:main}
There exists an $R(T)\bra{Q}$-algebra isomorphism 
\begin{equation} \label{eq:isomorphism}
\map[\xrightarrow{\sim}]{\Psi^{Q}}{(R(T)\bra{Q})[z_{1}^{\pm 1}, \ldots, z_{n}^{\pm 1}]/\CI^{Q}}{QK_{T}(G/B)}{}{z_{j} + \CI^{Q}}{\dfrac{1}{1-Q_{j}} [\CO_{G/B}(-\ve_{j})],}{\quad 1 \le l \le n.}
\end{equation}
\end{thm}

\section{\texorpdfstring{Equivariant $K$-groups of semi-infinite flag manifolds}{Equivariant K-groups of semi-infinite flag manifolds}} \label{sec:semi-infinite}
In this section, we recall the definition of the torus-equivariant $K$-group of semi-infinite flag manifolds, and its relationship with the torus-equivariant quantum $K$-ring. 
Also, we briefly review the inverse Chevalley formula for the torus-equivariant $K$-group of semi-infinite flag manifolds. 

\subsection{\texorpdfstring{Semi-infinite flag manifolds and the quantum $K$-ring}{Semi-infinite flag manifolds and the quantum K-ring}}
Following \cite[\S 3.1]{MNS} (which is based on \cite[\S1.4 and \S 1.5]{Kato}), we recall semi-infinite flag manifolds and their equivariant $K$-groups. 
Let $G$ be a connected, simply-connected, simple algebraic group over $\BC$, $B$ its Borel subgroup, $T \subset B$ a maximal torus, and $N$ the unipotent radical of $B$. 
The \emph{semi-infinite flag manifold} $\bQG^{\rat}$ is a (reduced) ind-scheme of ind-infinite type whose set of $\BC$-valued points is $G(\BC\pra{z})/(T(\BC) \cdot N(\BC\pra{z}))$. 

Let $\BI$ be the Iwahori subgroup of $G(\BC\bra{z})$, which is the preimage of $B$ under the evaluation map $G(\BC\bra{z}) \rightarrow G$, $z \mapsto 0$. Then, under the action of $\BI$ on $\bQG^{\rat}$, 
the set of $\BI$-orbits are labeled by the affine Weyl group $W_{\af} = \{wt_{\xi} \mid w \in W, \ \xi \in Q^{\vee}\} \simeq W \ltimes Q^{\vee}$ of $G$, where $t_{\xi}$ for $\xi \in Q^{\vee}$ denotes the translation element corresponding to $\xi$. 
For each $x \in W_{\af}^{\geq 0} := \{wt_{\xi} \in W_{\af} \mid w \in W, \ \xi \in Q^{\vee,+}\}$, we denote by $\bQG(x)$ the closure of the $\BI$-orbit labeled by $x$; 
this is a (reduced) closed subscheme of the (reduced) scheme $\bQG := \bQG{e}$ of infinite type, called the \emph{semi-infinite Schubert variety} associated to $x$, where $e \in W_{\af}$ is the identity element. 
For $x \in W_{\af}^{\geq 0}$, we denote by $\CO_{\bQG(x)}$ the structure sheaf of $\bQG(x)$. 
It is known that there exists a closed embedding $\bQG \subset \BP := \prod_{i \in I} \BP(L(\vpi_{i}) \otimes_{\BC} \BC\bra{z})$, where for a $\BC$-vector space $V$, $\BP(V)$ is the projective space associated to $V$, and $L(\lambda)$ for $\lambda \in P^{+}$ is the irreducible highest weight $G$-module of highest weight $\lambda$. 
We can define the line bundle $\CO_{\bQG}(\lambda)$ over $\bQG$ for each $\lambda = \sum_{i \in I} m_{i} \vpi_{i} \in P$ as the restriction of the line bundle $\boxtimes_{i \in I} \CO(m_{i})$ over $\BP$; 
for $\lambda, \mu \in P$, we have $\CO_{\bQG}(\lambda) \otimes \CO_{\bQG}(\mu) = \CO_{\bQG}(\lambda + \mu)$. 

Let $\KTCQG$ denote the \emph{$(T \times \BC^{\ast})$-equivariant $K$-group} of $\bQG$, where the action of $\BC^{\ast}$ on $\bQG$ is induced from the loop rotation action $\BC^{\ast} \curvearrowright \BC\bra{z}$; we denote by $q$ the character of the loop rotation action. More precisely, $\KTCQG$ is the $\BZ[q, q^{-1}][P]$-module consisting of all ``convergent'' infinite linear combinations of the elements $[\CO_{\bQG(x)}]$, $x \in W_{\af}^{\ge 0}$, called the \emph{semi-infinite Schubert classes}; 
the precise meaning of ``convergence'' is described in \cite[Proposition~5.8]{KNS}. 
Note that we can define the tensor product $[\CO_{\bQG(x)}] \otimes [\CO_{\bQG}(\lambda)] \in K_{T \times \BC^{\ast}}(\bQG)$ for $x \in W_{\af}^{\ge 0}$ and $\lambda \in P$; for $\lambda \in P$, we have the class $[\CO_{\bQG}(\lambda)] = 
[\CO_{\bQG(e)}] \otimes [\CO_{\bQG}(\lambda)] \in K_{T \times \BC^{\ast}}(\bQG)$ .
Also, by taking the specialization of $\KTCQG$ at $q = 1$, we obtain the \emph{$T$-equivariant $K$-group} $\KTQG$ of $\bQG$, which turns out to be the $R(T)$-module consisting of all infinite linear combinations of the \emph{semi-infinite Schubert classes} $[\CO_{\bQG(x)}]$, $x \in W_{\af}^{\ge 0}$, with coefficients in $R(T)$; 
that is, we can uniquely write each element of $\KTQG$ as an infinite linear combination of the semi-infinite Schubert classes $[\CO_{\bQG(x)}]$, $x \in W_{\af}^{\ge 0}$, with coefficients in $R(T)$. 
It follows that $[\CO_{\bQG(x)}] \otimes [\CO_{\bQG}(\lambda)] \in \KTQG$ for $\lambda \in P$ and $x \in W_{\af}^{\ge 0}$; for $\lambda \in P$, we have the class $[\CO_{\bQG}(\lambda)] = 
[\CO_{\bQG(e)}] \otimes [\CO_{\bQG}(\lambda)] \in \KTQG$. 

\begin{thm}[{\cite[Theorems~3.11 and 4.17]{Kato}}] \label{thm:quantum=semi-infinite}
There exists an isomorphism 
\begin{equation}
\Phi: QK_{T}(G/B) \xrightarrow{\sim} \KTQG
\end{equation}
of $R(T)$-modules such that 
\begin{itemize}
\item $\Phi(\e^{\mu}Q^{\xi}[\CO^{w}]) = \e^{-\mu}[\CO_{\bQG(wt_{\xi})}]$ for $\mu \in P$, $\xi \in Q^{\vee, +}$, and $w \in W$; 
\item for $\CZ \in QK_{T}(G/B)$, we have 
\begin{equation} \label{eq:compatible_anti-dominant}
\Phi(\CZ \star [\CO_{G/B}(-\vpi_{i})]) = \Phi(\CZ) \otimes [\CO_{\bQG}(w_{\circ}\vpi_{i})], \quad i \in I, 
\end{equation}
where $w_{\circ}$ denotes the longest element of $W$.
\end{itemize} 
\end{thm}
We warn the reader that the convention for line bundles in this paper is different from that in \cite{Kato} by the twist coming from the involution $-w_{\circ}$, and so we need $w_{\circ}$ on the right-hand side of \eqref{eq:compatible_anti-dominant}. 
If $G$ is of type $C$, then $w_{\circ}\lambda = -\lambda$ for all $\lambda \in P$, and hence 
the right-hand side of \eqref{eq:compatible_anti-dominant} is $\Phi(\CZ) \otimes [\CO_{\bQG}(-\vpi_{i})]$. 

Note that for a general $\lambda \in P$, the $R(T)$-module isomorphism $\Phi$ is not necessarily compatible with the quantum product $\star$ with $\CO_{G/B}(\lambda)$ in $QK_{T}(G/B)$ and the tensor product $\otimes$ with $\CO_{\bQG}(\lambda)$ in $\KTQG$. 
However, if $G$ is of type $C$, i.e., if $G = \Sp$, then we have the following equalities, which play an important role in this paper. 
\begin{prop}[{\cite[Proposition~5.2]{MNS}}] \label{prop:line_bundle_correspondence}
Let $G = \Sp$ be the symplectic group of rank $n$. For $\CZ \in QK_{T}(G/B)$, we have 
\begin{align}
\Phi \left( \mathcal{Z} \star \left( \frac{1}{1-Q_{j}} [\CO_{G/B}(\ve_{j})] \right) \right) &= \Phi(\mathcal{Z}) \otimes [\CO_{\bQG}(\ve_{j})], \\ 
\Phi \left( \mathcal{Z} \star \left( \frac{1}{1-Q_{j-1}} [\CO_{G/B}(-\ve_{j})] \right) \right) &= \Phi(\mathcal{Z}) \otimes [\CO_{\bQG}(-\ve_{j})] 
\end{align}
for $1 \le j \le n$. 
\end{prop}

For each $i \in I$, we define an $R(T)$-linear endomorphism $\st_{i}$ of $\KTQG$ by 
\begin{equation}
\st_{i}([\CO_{\bQG(x)}]) := \left[\CO_{\bQG\left(xt_{\alpha_{i}^{\vee}}\right)}\right]
\end{equation}
for $x \in W_{\af}^{\ge 0}$. Then, for $i \in I$, $\FZ \in \KTQG$, and $\lambda \in P$, we have 
\begin{equation}
\st_{i}(\FZ) \otimes [\CO_{\bQG}(\lambda)] = \st_{i}(\FZ \otimes [\CO_{\bQG}(\lambda)]) 
\end{equation}
(see \cite[(3.2)]{MNS}). 

\subsection{\texorpdfstring{Inverse Chevalley formula in type $C$}{Inverse Chevalley formula in type C}}
The \emph{inverse Chevalley formula} is an explicit formula for the expansion of the product $\e^{\mu} [\CO_{\bQG(x)}]$ for $\mu \in P$ and $x \in W_{\af}^{\ge 0}$ as a linear combination of the elements $[\CO_{\bQG(y)}] \otimes [\CO_{\bQG}(\lambda)]$, $y \in W_{\af}^{\ge 0}$, $\lambda \in P$, with coefficients in $\BZ[q, q^{-1}]$; namely, the inverse Chevalley formula is an explicit formula of the following form: 
\begin{equation}
\e^{\mu} [\CO_{\bQG(x)}] = \sum_{y \in W_{\af}^{\ge 0}, \, \lambda \in P} d_{x, \mu}^{y, \mu} [\CO_{\bQG(y)}] \otimes [\CO_{\bQG}(\lambda)], 
\end{equation}
with $d_{x, \mu}^{y, \mu} \in \BZ[q, q^{-1}]$. 
For $x \in W_{\af}^{\ge 0}$ and $\lambda \in P$, we set $[\CO_{\bQG(x)}(\lambda)] := [\CO_{\bQG(x)}] \otimes [\CO_{\bQG}(\lambda)]$. 
In types $A, D, E$, the inverse Chevalley formula in the general case is given by \cite{KNOS} and \cite{LNOS}, and in type $C$, it is given by \cite{KNO} in some special cases. 
In this paper, we use the inverse Chevalley formula in type $C$. 

Assume that $G$ is of type $C_{n}$, i.e., $G = \Sp$. 
From \cite[Theorem~4.5]{KNO}, we obtain the following. 
\begin{thm} \label{thm:IC_1}
Let $1\le k \le n-1$. The following equality holds in $\KTCQG$: 
\begin{equation} \label{eq:IC_1}
\begin{split}
\e^{\ve_{1}}[\CO_{\bQG(s_{1}s_{2} \cdots s_{k})}] &= [\CO_{\bQG(s_{1}s_{2} \cdots s_{k})}(\ve_{k+1})] - [\CO_{\bQG(s_{1}s_{2} \cdots s_{k+1})}(\ve_{k+1})] \\ 
& \quad + q \sum_{j = 1}^{k} \left[\CO_{\bQG\left(s_{1}s_{2} \cdots s_{j-1} t_{\alpha_{j}^{\vee} + \alpha_{j+1}^{\vee} + \cdots + \alpha_{k}^{\vee}}\right)}(\ve_{j})\right] \\ 
& \quad - q \sum_{j = 1}^{k} \left[\CO_{\bQG\left(s_{1}s_{2} \cdots s_{j} t_{\alpha_{j}^{\vee} + \alpha_{j+1}^{\vee} + \cdots + \alpha_{k}^{\vee}}\right)} (\ve_{j})\right]. 
\end{split}
\end{equation}
\end{thm}
Also, from \cite[Theorem~4.3]{KNO}, we obtain the following. 
\begin{thm} \label{thm:IC_2}
Let $1 \le k \le n$. The following equality holds in $\KTCQG$: 
\begin{equation} \label{eq:IC_2}
\begin{split}
& \e^{\ve_{1}}[\CO_{\bQG(s_{1} \cdots s_{n-1}s_{n}s_{n-1} \cdots s_{k})}] \\ 
& = [\CO_{\bQG(s_{1} \cdots s_{n-1}s_{n}s_{n-1} \cdots s_{k})}(-\ve_{k})] - [\CO_{\bQG(s_{1} \cdots s_{n-1}s_{n}s_{n-1} \cdots s_{k-1})}(-\ve_{k})] \\ 
& \quad + q \sum_{j = k+1}^{n} \left[\CO_{\bQG\left(s_{1} \cdots s_{n-1}s_{n}s_{n-1} \cdots s_{j} t_{\alpha_{k}^{\vee} + \alpha_{k+1}^{\vee} + \cdots + \alpha_{j-1}^{\vee}}\right)}(-\ve_{j})\right] \\ 
& \quad - q \sum_{j = k+1}^{n} \left[\CO_{\bQG\left(s_{1} \cdots s_{n-1}s_{n}s_{n-1} \cdots s_{j-1} t_{\alpha_{k}^{\vee} + \alpha_{k+1}^{\vee} + \cdots + \alpha_{j-1}^{\vee}}\right)}(-\ve_{j})\right] \\ 
& \quad + q \sum_{j = 1}^{k} \left[ \CO_{\bQG\left(s_{1}s_{2} \cdots s_{j-1} t_{\alpha_{j}^{\vee} + \alpha_{j+1}^{\vee} + \cdots + \alpha_{n}^{\vee}}\right)}(\ve_{j}) \right] \\ 
& \quad - q \sum_{j = 1}^{k} \left[ \CO_{\bQG\left(s_{1}s_{2} \cdots s_{j} t_{\alpha_{j}^{\vee} + \alpha_{j+1}^{\vee} + \cdots + \alpha_{n}^{\vee}}\right)}(\ve_{j}) \right]; 
\end{split}
\end{equation}
when $k = 1$, the second term $[\CO_{\bQG(s_{1} \cdots s_{n-1}s_{n}s_{n-1} \cdots s_{k-1})}(-\ve_{k})]$ on the right-hand side of the above formula is understood to be $0$. 
\end{thm}

We defer the proofs of \eqref{eq:IC_1} and \eqref{eq:IC_2} to Appendix~\ref{sec:IC_proof}. 

\subsection{Demazure operators}
Following \cite[\S A]{MNS}, we review Demazure operators acting on $\KTCQG$ and $\KTQG$. 
Let $G$ be of an arbitrary type. 
The \emph{nil-DAHA} $\BH_{0}$ is defined to be the $\BZ[\sq^{\pm 1}]$-algebra generated by $\ST_{i}$, $i \in I \sqcup \{0\}$, and $\SX^{\nu}$, $\nu \in P$, subject to certain defining relations (including the braid relations); 
see, for example, \cite[(A.2)--(A.6)]{MNS}). 
For $i \in I \sqcup \{0\}$, we set  $\SD_{i} := 1 + \ST_{i}$ for $i \in I \sqcup \{0\}$. 

By \cite[Theorem~6.5]{KNS} (see also \cite[\S 2.6]{Orr} and \cite[\S 3.1.2]{KNOS}), we have a left $\BH_{0}$-action on $K_{T \times \BC^{*}}(\bQG^{\rat})$. 
In particular, for $i \in I$, we have an action of $\SD_{i}$ on $\KTCQG$, which is called a Demazure operator; note that these Demazure operators satisfy the braid relations. 
Let $D_{i}$ (not to be confused with $\SD_{i}$) for $i \in I$ be an operator on $\BZ[P]$ given by 
\begin{equation}
D_{i}(\e^{\nu}) := \frac{\e^{\nu} - \e^{\alpha_{i}}\e^{s_{i}\nu}}{1-\e^{\alpha_{i}}}
\end{equation}
for $\nu \in P$, i.e., 
\begin{equation}
D_{i}(\e^{\nu}) = \begin{cases}
\e^{\nu}(1 + \e^{\alpha_{i}} + \e^{2\alpha_{i}} + \cdots + \e^{-\pair{\nu}{\alpha_{i}^{\vee}}\alpha_{i}}) & \text{if $\pair{\nu}{\alpha_{i}^{\vee}} \le 0$}, \\ 
0 & \text{if $\pair{\nu}{\alpha_{i}^{\vee}} = 1$}, \\ 
-\e^{\nu}(\e^{-\alpha_{i}} + \e^{-2\alpha_{i}^{\vee}} + \cdots + \e^{-(\pair{\nu}{\alpha_{i}^{\vee}}-1)\alpha_{i}}) & \text{if $\pair{\nu}{\alpha_{i}^{\vee}} \ge 2$}. 
\end{cases}
\end{equation}
\begin{lem}[{\cite[Lemma~A.2]{MNS}; \cite[\S 2.6]{Orr} and \cite[\S 3.1.2]{KNOS}}]
Let $i \in I$, and let $\xi \in Q^{\vee,+}$, $\nu, \lambda \in P$. 
In $\KTCQG$, we have 
\begin{equation}
\SD_{i}(\e^{\nu}[\CO_{\bQG(t_{\xi})}(\lambda)]) = D_{i}(\e^{\nu}) [\CO_{\bQG(t_{\xi})}(\lambda)]. 
\end{equation}
\end{lem}

\section{\texorpdfstring{Key relations in $\KTQG$}{Key relations in KT(QG)}} \label{sec:key}
In order to prove Theorem~\ref{thm:main}, we need to obtain sufficiently many relations in $\KTQG$. 
First, by making use of the inverse Chevalley formula (see \eqref{eq:IC_1} and \eqref{eq:IC_2}), we obtain a ``base'' relation. 
Then, by applying Demazure operators to the ``base'' relation, we obtain the other relations. 
These relations can be thought of as recurrence relations for the elements $\FF_{l}$, $0 \le l \le n$, of $K_{T}(\bQG)$ (defined in Definition~\ref{def:FFl}) with coefficients in $R(T)$, which uniquely determine the elements $\FF_{l}$, $0 \le l \le n$. Hence, by solving these recurrence relations, 
we obtain an explicit description (see Theorem~\ref{thm:relation_semi-infinite}) of the elements $\FF_{l}$, $0 \le l \le n$, as explicit elements of $R(T)$. 
Although this strategy for obtaining the desired description of $\FF_{l}$, $0 \le l \le n$, is almost the same as that in \cite{MNS} in type $A$ , 
we need to overcome some technical difficulties peculiar to the case of type $C$; in particular, we need Proposition~\ref{prop:symmetry}. 

In this section, we assume that $G = \Sp$, the symplectic group of rank $n$. 

\subsection{``Base'' relation}
From equations \eqref{eq:IC_1} and \eqref{eq:IC_2}, which are special cases of the inverse Chevalley formula, we will obtain a ``base'' relation in $\KTQG$. 
First, by multiplying both sides of \eqref{eq:IC_1} by $[\CO_{\bQG}(-\ve_{k+1})]$ and specializing $q$ to $1$, we obtain the following. 
\begin{lem}
Let $1 \le k \le n-1$. The following equality holds in $\KTQG$: 
\begin{equation} \label{eq:rec_1}
\begin{split}
[\CO_{\bQG(s_{1}s_{2} \cdots s_{k+1})}] &= -\e^{\ve_{1}} [\CO_{\bQG(s_{1}s_{2} \cdots s_{k})}(-\ve_{k+1})] + [\CO_{\bQG(s_{1}s_{2} \cdots s_{k})}] \\ 
& \quad + \sum_{j = 1}^{k} \left[\CO_{\bQG\left(s_{1}s_{2} \cdots s_{j-1} t_{\alpha_{j}^{\vee} + \alpha_{j+1}^{\vee} + \cdots + \alpha_{k}^{\vee}}\right)}(\ve_{j} - \ve_{k+1}) \right] \\ 
& \quad - \sum_{j = 1}^{k} \left[\CO_{\bQG\left(s_{1}s_{2} \cdots s_{j} t_{\alpha_{j}^{\vee} + \alpha_{j+1}^{\vee} + \cdots + \alpha_{k}^{\vee}}\right)}(\ve_{j} - \ve_{k+1}) \right]. 
\end{split}
\end{equation}
\end{lem}

Also, by multiplying both sides of \eqref{eq:IC_2} by $[\CO_{\bQG}(\ve_{k})]$ and specializing $q$ to $1$, we obtain the following. 
\begin{lem}
Let $1 \le k \le n$. The following equality holds in $\KTQG$: 
\begin{equation} \label{eq:rec_2}
\begin{split}
& [\CO_{\bQG(s_{1} \cdots s_{n-1}s_{n}s_{n-1} \cdots s_{k-1})}] \\ 
&= -\e^{\ve_{1}}[\CO_{\bQG(s_{1} \cdots s_{n-1}s_{n}s_{n-1} \cdots s_{k})}(\ve_{k})] + [\CO_{\bQG(s_{1} \cdots s_{n-1}s_{n}s_{n-1} \cdots s_{k})}] \\ 
& \quad + \sum_{j = k+1}^{n} \left[\CO_{\bQG\left(s_{1} \cdots s_{n-1}s_{n}s_{n-1} \cdots s_{j} t_{\alpha_{k}^{\vee} + \alpha_{k+1}^{\vee} + \cdots + \alpha_{j-1}^{\vee}}\right)}(\ve_{k} - \ve_{j})\right] \\ 
& \quad - \sum_{j = k+1}^{n} \left[\CO_{\bQG\left(s_{1} \cdots s_{n-1}s_{n}s_{n-1} \cdots s_{j-1} t_{\alpha_{k}^{\vee} + \alpha_{k+1}^{\vee} + \cdots + \alpha_{j-1}^{\vee}}\right)}(\ve_{k} - \ve_{j})\right] \\ 
& \quad + \sum_{j = 1}^{k} \left[\CO_{\bQG\left(s_{1}s_{2} \cdots s_{j-1} t_{\alpha_{j}^{\vee} + \alpha_{j+1}^{\vee} + \cdots + \alpha_{n}^{\vee}}\right)}(\ve_{j}+\ve_{k}) \right]\\ 
& \quad - \sum_{j = 1}^{k} \left[\CO_{\bQG\left(s_{1}s_{2} \cdots s_{j} t_{\alpha_{j}^{\vee} + \alpha_{j+1}^{\vee} + \cdots + \alpha_{n}^{\vee}}\right)}(\ve_{j}+\ve_{k}) \right]; 
\end{split}
\end{equation}
when $k = 1$, the left-hand side of the above formula is understood to be $0$. 
\end{lem}

For $1 \le k \le n$, we set 
\begin{align}
\FP_{k} &:= [\CO_{\bQG(s_{1}s_{2} \cdots s_{k})}], \\ 
\FQ_{k} &:= [\CO_{\bQG(s_{1} \cdots s_{n-1}s_{n}s_{n-1} \cdots s_{k})}]. 
\end{align}
Also, we set $\FP_{0} := 1$ and $\FQ_{0} := 0$. 
Then, \eqref{eq:rec_1} and \eqref{eq:rec_2} can be rewritten as: 
\begin{equation} \label{eq:rec_3}
\begin{split}
\FP_{k+1} &= -\e^{\ve_{1}}\FP_{k} \otimes [\CO_{\bQG}(-\ve_{k+1})] + \FP_{k} \\ 
& \quad + \sum_{j = 1}^{k} \st_{j}\st_{j+1} \cdots \st_{k} \FP_{j-1} \otimes [\CO_{\bQG}(\ve_{j} - \ve_{k+1})] \\ 
& \quad - \sum_{j = 1}^{k} \st_{j}\st_{j+1} \cdots \st_{k} \FP_{j} \otimes [\CO_{\bQG}(\ve_{j} - \ve_{k+1})], 
\end{split}
\end{equation}
\begin{equation} \label{eq:rec_4}
\begin{split}
\FQ_{k-1} &= -\e^{\ve_{1}}\FQ_{k} \otimes [\CO_{\bQG}(\ve_{k})] + \FQ_{k} \\ 
& \quad + \sum_{j = k+1}^{n} \st_{k}\st_{k+1} \cdots \st_{j-1} \FQ_{j} \otimes [\CO_{\bQG}(\ve_{k} - \ve_{j})] \\ 
& \quad - \sum_{j = k+1}^{n} \st_{k}\st_{k+1} \cdots \st_{j-1} \FQ_{j-1} \otimes [\CO_{\bQG}(\ve_{k} - \ve_{j})] \\ 
& \quad + \sum_{j = 1}^{k} \st_{j}\st_{j+1} \cdots \st_{n} \FP_{j-1} \otimes [\CO_{\bQG}(\ve_{j}+\ve_{k})] \\ 
& \quad - \sum_{j = 1}^{k} \st_{j}\st_{j+1} \cdots \st_{n} \FP_{j} \otimes [\CO_{\bQG}(\ve_{j}+\ve_{k})]. 
\end{split}
\end{equation}
Equations \eqref{eq:rec_3} and \eqref{eq:rec_4} can be thought of as recurrence relations for $\FP_{0}$, $\FP_{1}$, $\FP_{2}$, $\ldots$, $\FP_{n}$, $\FQ_{n}$, $\FQ_{n-1}$, $\ldots$, $\FQ_{1}$, $\FQ_{0}$. 
By solving these recurrence relations, we can write these elements as explicit linear combinations of line bundle classes with coefficients in $\End_{\BC}(K_{T}(\bQG))$. 

\begin{defn}
Let $I \subset [1, \overline{1}]$. 
\begin{enu}
\item For $1 \le j \le n$, we define $\psi_{I}(j) \in \End_{\BC}(K_{T}(\bQG))$ by 
\begin{equation}
\psi_{I}(j) := \begin{cases}
1-\st_{j} & \text{if $j \notin I$ and $j+1 \in I$}, \\ 
1 & \text{otherwise}, 
\end{cases}
\end{equation}
where we understand that $n+1 := \overline{n}$. 

\item For $2 \le j \le n$, we define $\psi_{I}(\overline{j}) \in \End_{\BC}(K_{T}(\bQG))$ by 
\begin{equation}
\psi_{I}(\overline{j}) := \begin{cases} 
1 - \st_{j-1} + \st_{j-1} \st_{j} \cdots \st_{n} & \text{if $I$ is of the form $\{ \cdots < j-1 < \overline{j-1} < \cdots \}$,} \\ 
1 - \st_{j-1} & \text{\parbox{12em}{if $I$ is not of the above form \\ and $\overline{j} \notin I$, $\overline{j-1} \in I$,}} \\ 
1 & \text{otherwise}. \end{cases}
\end{equation}

\item Define $\psi_{I}(\overline{1}) := 1 \in \End_{\BC}(K_{T}(\bQG))$. 
\end{enu}
\end{defn}

\begin{exm}
Let $n = 4$ and $I = \{2, 3, \overline{3}, \overline{1}\}$. Then the elements $\psi_{I}(j)$ for $1 \le j \le \overline{1}$ are as in the following Table~\ref{tab:list_example}. 
\begin{table}[h]
\centering
\caption{The list of $\psi_{I}(j)$ for $1 \le j \le \overline{1}$}
\label{tab:list_example}
\begin{tabular}{|c|cccccccc|} \hline
$j$ & $1$ & $2$ & $3$ & $4$ & $\rule{0em}{2.5ex}\overline{4}$ & $\overline{3}$ & $\overline{2}$ & $\overline{1}$ \\ \hline 
$\psi_{I}(j)$ & $1-\st_{1}$ & $1$ & $1$ & $1$ & $1 - \st_{3} + \st_{3}\st_{4}$ & $1$ & $1 - \st_{1}$ & $1$ \\ \hline
\end{tabular}
\end{table}
\end{exm}

For $J \subset [1, \overline{1}]$, we set 
\begin{equation}
\ve_{J} := \sum_{j \in I} \ve_{j}. 
\end{equation}

\begin{defn} \label{def:FFl}
\begin{enu}
\item For $1 \le k \le n$ and $0 \le l \le k$, we define $\FF_{l}^{k} \in \KTQG$ by 
\begin{equation}
\FF_{l}^{k} := \sum_{\substack{I \subset [1, k] \\ |I| = l}} \left( \prod_{1 \le j \le \overline{1}} \psi_{I}(j) \right) [\CO_{\bQG}(-\ve_{I})]. 
\end{equation}

\item For $1 \le k \le n$ and $0 \le l \le 2n-k$, we define $\FF_{l}^{\overline{k}} \in \KTQG$ by 
\begin{equation}
\FF_{l}^{\overline{k}} := \sum_{\substack{I \subset [1, \overline{k+1}] \\ |I| = l}} \left( \prod_{1 \le j \le \overline{1}} \psi_{I}(j) \right) [\CO_{\bQG}(-\ve_{I})], 
\end{equation}
where we understand that $\overline{n+1} := n$. 

\item For $0 \le l \le 2n$, we define $\FF_{l} \in \KTQG$ by 
\begin{equation}
\FF_{l} := \sum_{\substack{I \subset [1, \overline{1}] \\ |I| = l}} \left( \prod_{1 \le j \le \overline{1}} \psi_{I}(j) \right) [\CO_{\bQG}(-\ve_{I})]. 
\end{equation}
\end{enu}
\end{defn}

\begin{thm} \label{thm:solution_recursion}
\begin{enu}
\item For $1 \le k \le n$, we have 
\begin{equation}
\FP_{k} = \sum_{l = 0}^{k} (-1)^{l} \e^{l\ve_{1}} \FF_{l}^{k}. 
\end{equation}

\item For $1 \le k \le n$, we have 
\begin{equation} \label{eq:solve_2}
\FQ_{k} = \sum_{l = 0}^{2n-k} (-1)^{l} \e^{l\ve_{1}} \FF_{l}^{\overline{k}}. 
\end{equation}

\item We have 
\begin{equation} \label{eq:solve_3}
\sum_{l = 0}^{2n} (-1)^{l} \e^{l\ve_{1}} \FF_{l} = 0. 
\end{equation}
\end{enu}
\end{thm}

\begin{proof}
Part (1) can be proved by the same argument as for \cite[Proposition~4.7]{MNS}. 
If we set $\FF_{l}^{\overline{0}} := \FF_{l}$ for $0 \le l \le 2n$, then part (3) can be regarded as the special case $k = 0$ of \eqref{eq:solve_2}. 
Hence, in order to prove parts (2) and (3), it suffices to prove \eqref{eq:solve_2} for $0 \le k \le n$.
We will prove \eqref{eq:solve_2} by downward induction on $k = n, n-1, \ldots, 1, 0$. 
Since $\FP_{n} = \FQ_{n}$, the case $k = n$ is already proved by part (1). 
Now assume that \eqref{eq:solve_2} holds for all $k$, with $0 < k \le n$. 
By \eqref{eq:rec_4}, we see that 
\begin{align}
\begin{split}
\FQ_{k-1} &= -\e^{\ve_{1}} \sum_{l = 0}^{2n-k} (-1)^{l} \e^{l\ve_{1}} \FF_{l}^{\overline{k}} \otimes [\CO_{\bQG}(\ve_{k})] + \sum_{l = 0}^{2n-k} (-1)^{l} \e^{l\ve_{1}} \FF_{l}^{\overline{k}} \\ 
& \quad + \sum_{i = k+1}^{n} \st_{k}\st_{k+1} \cdots \st_{i-1} \sum_{l = 0}^{2n-i} (-1)^{l} \e^{l\ve_{1}} \FF_{l}^{\overline{i}} \otimes [\CO_{\bQG}(\ve_{k} - \ve_{i})] \\ 
& \quad - \sum_{i = k+1}^{n} \st_{k}\st_{k+1} \cdots \st_{i-1} \sum_{l = 0}^{2n-i+1} (-1)^{l} \e^{l\ve_{1}} \FF_{l}^{\overline{i-1}} \otimes [\CO_{\bQG}(\ve_{k} - \ve_{i})] \\ 
& \quad + \sum_{i = 1}^{k} \st_{i}\st_{i+1} \cdots \st_{n} \sum_{l = 0}^{i-1} (-1)^{l} \e^{l\ve_{1}} \FF_{l}^{i-1} \otimes [\CO_{\bQG}(\ve_{i}+\ve_{k})] \\ 
& \quad - \sum_{i = 1}^{k} \st_{i}\st_{i+1} \cdots \st_{n} \sum_{l = 0}^{i} (-1)^{l} \e^{l\ve_{1}} \FF_{l}^{i} \otimes [\CO_{\bQG}(\ve_{i}+\ve_{k})] 
\end{split} \\ 
\begin{split}
&= \sum_{l = 0}^{2n-k} (-1)^{l+1} \e^{(l+1)\ve_{1}} \sum_{\substack{I \subset [1, \overline{k+1}] \\ |I| = l}} \left( \prod_{1 \le j \le \overline{1}} \psi_{I}(j) \right) [\CO_{\bQG}(-\ve_{I} + \ve_{k})] \\ 
& \quad + \sum_{l = 0}^{2n-k} (-1)^{l} \e^{l\ve_{1}} \sum_{\substack{I \subset [1, \overline{k+1}] \\ |I| = l}} \left( \prod_{1 \le j \le \overline{1}} \psi_{I}(j) \right) [\CO_{\bQG}(-\ve_{I})] \\ 
& \quad + \sum_{i = k+1}^{n} \st_{k}\st_{k+1} \cdots \st_{i-1} \sum_{l = 0}^{2n-i} (-1)^{l} \e^{l\ve_{1}} \sum_{\substack{I \subset [1, \overline{i+1}] \\ |I| = l}} \left( \prod_{1 \le j \le \overline{1}} \psi_{I}(j) \right) [\CO_{\bQG}(-\ve_{I} + \ve_{k} - \ve_{i})] \\ 
& \quad - \sum_{i = k+1}^{n} \st_{k}\st_{k+1} \cdots \st_{i-1} \sum_{l = 0}^{2n-i+1} (-1)^{l} \e^{l\ve_{1}} \sum_{\substack{I \subset [1, \overline{i}] \\ |I| = l}} \left( \prod_{1 \le j \le \overline{1}} \psi_{I}(j) \right) [\CO_{\bQG}(-\ve_{I} + \ve_{k} - \ve_{i})] \\ 
& \quad + \sum_{i = 1}^{k} \st_{i}\st_{i+1} \cdots \st_{n} \sum_{l = 0}^{i-1} (-1)^{l} \e^{l\ve_{1}} \sum_{\substack{I \subset [1, i-1] \\ |I| = l}} \left( \prod_{1 \le j \le \overline{1}} \psi_{I}(j) \right) [\CO_{\bQG}(-\ve_{I}+\ve_{i}+\ve_{k})] \\ 
& \quad - \sum_{i = 1}^{k} \st_{i}\st_{i+1} \cdots \st_{n} \sum_{l = 0}^{i} (-1)^{l} \e^{l\ve_{1}} \sum_{\substack{I \subset [1, i] \\ |I| = l}} \left( \prod_{1 \le j \le \overline{1}} \psi_{I}(j) \right) [\CO_{\bQG}(-\ve_{I}+\ve_{i}+\ve_{k})] 
\end{split} \\ 
\begin{split}
&= \sum_{I \subset [1, \overline{k+1}]}  (-1)^{|I|+1} \e^{(|I|+1)\ve_{1}} \left( \prod_{1 \le j \le \overline{1}} \psi_{I}(j) \right) [\CO_{\bQG}(-\ve_{I}+\ve_{k})] \\ 
& \quad + \sum_{I \subset [1, \overline{k+1}]}  (-1)^{|I|} \e^{|I|\ve_{1}} \left( \prod_{1 \le j \le \overline{1}} \psi_{I}(j) \right) [\CO_{\bQG}(-\ve_{I})] \\ 
& \quad - \sum_{i = k+1}^{n} \st_{k}\st_{k+1} \cdots \st_{i-1} \sum_{\substack{I \subset [1, \overline{i}] \\ \overline{i} \in I}} (-1)^{|I|} \e^{|I|\ve_{1}} \left( \prod_{1 \le j \le \overline{1}} \psi_{I}(j) \right) [\CO_{\bQG}(-\ve_{I}+\ve_{k} - \ve_{i})] \\ 
& \quad - \sum_{i = 1}^{k} \st_{i}\st_{i+1} \cdots \st_{n} \sum_{\substack{I \subset [1, i] \\ i \in I}} (-1)^{|I|} \e^{|I|\ve_{1}} \left( \prod_{1 \le j \le \overline{1}} \psi_{I}(j) \right) [\CO_{\bQG}(-\ve_{I}+\ve_{i}+\ve_{k})]. 
\end{split} \label{eq:recursive}
\end{align}
Let us compute the term corresponding to each $I \subset [1, \overline{k}]$ on the RHS of \eqref{eq:recursive}. 

Let $I \subset [1, \overline{k}]$ be such that $\overline{k} \notin I$; we set $l := |I|$. 
Then we see that the second sum on the RHS of \eqref{eq:recursive} contains the term
\begin{align}
& (-1)^{l} \e^{l\ve_{1}} \left( \prod_{1 \le j \le \overline{1}} \psi_{I}(j) \right) [\CO_{\bQG}(-\ve_{I})] \\ 
&= (-1)^{l} \e^{l\ve_{1}} \left( \prod_{1 \le j \le \overline{1}} \psi_{I}(j) \right) [\CO_{\bQG}(-\ve_{I})], 
\end{align}
which is just the term in $\FQ_{k-1}$ corresponding to $I$. 

Next, let $I \in [1, \overline{k}]$ be such that $\overline{k} \in I$; 
we set $M := \max (I \setminus \{\overline{k}\}) \in [1, \overline{1}]$. 
Assume that $M < n$. 
By summing up the terms: 
\begin{itemize}
\item the term in the first sum on the RHS of \eqref{eq:recursive} corresponding to $I \setminus \{\overline{k}\}$, 
\item the terms in the third sum on the RHS of \eqref{eq:recursive} corresponding to $(I \setminus \{\overline{k}\}) \sqcup \{\overline{i}\}$ for $k+1 \le i \le n$ and $j \in I$, 
\item the terms in the fourth sum on the RHS of \eqref{eq:recursive} corresponding to $(I \setminus \{\overline{k}\}) \sqcup \{i\}$ for $M+1 \le i \le k$ (if $k \le M$, then ignore this sum), 
\end{itemize}
we obtain the following: 
\begin{align}
\begin{split}
& (-1)^{(|I|-1) + 1} \e^{((|I|-1)+1)\ve_{1}} \left( \prod_{1 \le j \le \overline{1}} \psi_{I \setminus \{\overline{k}\}}(j) \right) [\CO_{\bQG}(-\ve_{I \setminus \{\overline{k}\}} + \ve_{k})] \\ 
&- \sum_{i = k+1}^{n} \st_{k}\cdots \st_{i-1} (-1)^{|I|} \e^{|I|\ve_{1}} \left( \prod_{1 \le j \le \overline{1}} \psi_{(I \setminus \{\overline{k}\}) \sqcup \{\overline{i}\}}(j) \right) [\CO_{\bQG}(-\ve_{(I \setminus \{\overline{k}\}) \sqcup \{\overline{i}\}}+\ve_{k}-\ve_{i})] \\ 
&- \sum_{i = M+1}^{k} \st_{i} \cdots \st_{n} (-1)^{|I|} \e^{|I|\ve_{1}} \left( \prod_{1 \le j \le \overline{1}} \psi_{(I \setminus \{\overline{k}\}) \sqcup \{ i\}}(j) \right) [\CO_{\bQG}(-\ve_{(I \setminus \{\overline{k}\}) \sqcup \{ i \}}+\ve_{i}+\ve_{k})]
\end{split} \\ 
\begin{split}
&= (-1)^{|I|} \e^{|I|\ve_{1}} \left( \left( \prod_{1 \le j \le \overline{1}} \psi_{I \setminus \{\overline{k}\}}(j) \right) - \sum_{i = k+1}^{n} \st_{k} \cdots \st_{i-1} \left( \prod_{1 \le j \le \overline{1}} \psi_{(I \setminus \{\overline{k}\}) \sqcup \{\overline{i}\}}(j) \right) \right. \\ 
& \hspace{40mm} \left. - \sum_{i = M+1}^{k} \st_{i} \cdots \st_{n} \left( \prod_{1 \le j \le \overline{1}} \psi_{(I \setminus \{\overline{k}\}) \sqcup \{ i\}}(j) \right) \right)  [\CO_{\bQG}(-\ve_{I})]. 
\end{split} \\ 
\begin{split}
&= (-1)^{|I|} \e^{|I|\ve_{1}} \left( \prod_{1 \le j \le M} \psi_{I \setminus \{\overline{k}\}}(j) \right) \\ 
& \qquad \times \left( 1 - \sum_{i = k+1}^{n} \st_{k} \cdots \st_{i-1} \psi_{(I \setminus \{\overline{k}\}) \sqcup \{\overline{i}\}} (\overline{i+1}) \right. \\ 
& \hspace{30mm} \left. {}-{} \st_{M+1} \cdots \st_{n} - \sum_{i = M+2}^{k} \st_{i} \cdots \st_{n} \psi_{(I \setminus \{\overline{k}\}) \sqcup \{i\}}(i-1) \right) [\CO_{\bQG}(-\ve_{I})]. 
\end{split} \label{eq:recursion_middle}
\end{align}
Here we set 
\begin{equation} \label{eq:recursion_sub}
\varphi := 1 - \sum_{i = k+1}^{n} \st_{k} \cdots \st_{i-1} \psi_{(I \setminus \{\overline{k}\}) \sqcup \{\overline{i}\}} (\overline{i+1}) - \st_{M+1} \cdots \st_{n} - \sum_{i = M+2}^{k} \st_{i} \cdots \st_{n} \psi_{(I \setminus \{\overline{k}\}) \sqcup \{i\}}(i-1). 
\end{equation} 
If $M < k$, then 
\begin{align}
\begin{split}
\varphi &= 1 - \sum_{i = k+1}^{n} \st_{k} \cdots \st_{i-1} \psi_{(I \setminus \{\overline{k}\}) \sqcup \{\overline{i}\}} (\overline{i+1}) \\
& \qquad \quad - \st_{M+1} \cdots \st_{n} - \sum_{i = M+2}^{k} \st_{i} \cdots \st_{n} \psi_{(I \setminus \{\overline{k}\}) \sqcup \{i\}}(i-1)
\end{split} \\ 
&= 1 - \sum_{i = k+1}^{n} \st_{k} \cdots \st_{i-1} (1 - \st_{i}) - \st_{M+1} \cdots \st_{n} - \sum_{i = M+2}^{k} \st_{i} \cdots \st_{n} (1 - \st_{i-1}) \\ 
&= 1-\st_{k} \\ 
&= \psi_{I}(\overline{k+1}). 
\end{align}
If $M = k$, then 
\begin{align}
\varphi &= 1 - \sum_{i = k+1}^{n} \st_{k} \cdots \st_{i-1} \underbrace{\psi_{(I \setminus \{\overline{k}\}) \sqcup \{\overline{i}\}}(\overline{i+1})}_{= \, 1 - \st_{j}} \\ 
&= 1 - \st_{k} + \st_{k} \cdots \st_{n} \\ 
&= \psi_{I}(\overline{k+1}). 
\end{align}
If $M > k$, then 
\begin{align}
\varphi &= 1 - \sum_{i = k+1}^{n} \st_{k} \cdots \st_{i-1} \underbrace{\psi_{(I \setminus \{\overline{k}\}) \sqcup \{\overline{i}\}}(\overline{i+1})}_{\mathclap{= \, 1 - \st_{j} + \st_{j} \cdots \st_{n} \ \text{(if $j = M$)}, \ 1 - \st_{j} \ \text{(if $j \not= M$)}}} \\ 
&= 1 - \sum_{i = k+1}^{M-1} (1-\st_{i}) - \st_{k} \cdots \st_{M-1} (1 - \st_{M} + \st_{M} \cdots \st_{n}) - \sum_{i = M+1}^{n} \st_{k} \cdots \st_{j-1} (1 - \st_{i}) \\ 
&= 1 - \st_{k} \\
&= \psi_{I}(\overline{k+1}). 
\end{align}
From these, we conclude that the RHS of \eqref{eq:recursion_middle} equals 
\begin{equation}
(-1)^{|I|} \e^{|I|\ve_{1}} \left( \prod_{1 \le j \le 1} \psi_{I}(j) \right) [\CO_{\bQG}(-\ve_{I})], 
\end{equation}
which is just the term in $\FQ_{k-1}$ corresponding to $I$. 

Assume that $M \ge n$. If $M > n$, then take $1 \le m \le n$ such that $M = \overline{m}$. If $M = n$, then set $m := \overline{n}$. By summing up the terms: 
\begin{itemize}
\item the term in the first sum on the RHS of \eqref{eq:recursive} corresponding to $I \setminus \{\overline{k}\}$, 
\item the terms in the third sum on the RHS of \eqref{eq:recursive} corresponding to $(I \setminus \{\overline{k}\}) \sqcup \{\overline{i}\}$ for $k+1 \le i \le m-1$ (if $m = \overline{n}$, then regard $\overline{n}-1$ as $n$), 
\end{itemize}
we obtain the following (if $m = \overline{n}$, then regard $\overline{n}-2$ as $n-1$): 
\begin{align}
\begin{split}
& (-1)^{(|I|-1)+1} \e^{((|I|-1)+1)\ve_{1}} \left( \prod_{1 \le j \le \overline{1}} \psi_{I \setminus \{\overline{k}\}}(j) \right) [\CO_{\bQG}(-\ve_{I \setminus \{\overline{k}\}}+\ve_{k})] \\ 
& - \sum_{i = k+1}^{m-1} \st_{k} \cdots \st_{i-1} (-1)^{|I|} \e^{|I|\ve_{1}} \left( \prod_{1 \le j \le \overline{1}} \psi_{(I \setminus \{\overline{k}\}) \sqcup \{\overline{i}\}}(j) \right) [\CO_{\bQG}(-\ve_{(I \setminus \{\overline{k}\}) \sqcup \{\overline{i}\}}+\ve_{k} - \ve_{i})]
\end{split} \\ 
&= (-1)^{|I|} \e^{|I|\ve_{1}} \left( \left( \prod_{1 \le j \le M} \psi_{I \setminus \{\overline{k}\}}(j) \right) \left(1 - \st_{k} \cdots \st_{m-2} \right. \right. \\ 
& \hspace{40mm} \left. \left. - \sum_{i = k+1}^{m-2} \st_{k} \cdots \st_{i-1} \underbrace{\psi_{(I \setminus \{\overline{k}\}) \sqcup \{\overline{i}\}}(\overline{i+1})}_{= \, 1 - \st_{i}} \right) \right) [\CO_{\bQG}(-\ve_{I})] \\ 
&= (-1)^{|I|} \e^{|I|\ve_{1}} \left( \prod_{1 \le j \le M} \psi_{I \setminus \{\overline{k}\}}(j) \right) (1 - \st_{k}) [\CO_{\bQG}(-\ve_{I})] \\ 
&= (-1)^{|I|} \e^{|I|\ve_{1}} \left( \prod_{1 \le j \le 1} \psi_{I \setminus \{\overline{k}\}}(j) \right) [\CO_{\bQG}(-\ve_{I})], 
\end{align}
which is just the term in $\FQ_{k-1}$ corresponding to $I$. 
This proves \eqref{eq:solve_2} for $k-1$. 
Thus, by downward induction on $k$, equation \eqref{eq:solve_2} is proved. 
This completes the proof of the theorem. 
\end{proof}

Also, we can show the following; 
we defer the proof of this proposition to Appendix~\ref{sec:proof_symmetry}.  
\begin{prop} \label{prop:symmetry}
For $0 \le k \le n$, we have $\FF_{k} = \FF_{2n-k}$. 
\end{prop}

Therefore, by multiplying both sides of \eqref{eq:solve_3} by $\e^{-n\ve_{1}}$, we obtain the following ``base'' relation. 
\begin{cor}
The following equality holds in $\KTQG$: 
\begin{equation} \label{eq:initial}
\sum_{l = 0}^{n-1} (-1)^{l} (\e^{-(n-l)\ve_{1}} + \e^{(n-l)\ve_{1}}) \FF_{l} + (-1)^{n}\FF_{n} = 0. 
\end{equation}
\end{cor}

\subsection{The other relations}
From the base relation \eqref{eq:initial}, we will derive sufficiently many relations for a Borel-type presentation of $QK_{T}(G/B)$. 
For $k \ge 0$, let $h_{k}(x_{1}, \ldots, x_{2n})$ be the $k$-th complete symmetric polynomial in the ($2n$) variables $x_{1}, \ldots, x_{2n}$, i.e.,  
\begin{equation}
h_{k}(x_{1}, \ldots, x_{2n}) := \sum_{1 \le i_{1} \le \cdots \le i_{k} \le 2n} x_{i_{1}} x_{i_{2}} \cdots x_{i_{2n}}; 
\end{equation}
by convention, we set $h_{0}(x_{1}, \ldots, x_{2n}) := 0$. 
For $l \ge 0$ and $k \ge 1$, we set 
\begin{equation}
\FH_{l}^{k} := h_{l}(\e^{\ve_{1}}, \ldots, \e^{\ve_{k-1}}, \e^{\ve_{k}}, \e^{-\ve_{k}}, \e^{-\ve_{k-1}}, \ldots, \e^{-\ve_{1}}) \in R(T). 
\end{equation}

The aim of this subsection is to prove the following.

\begin{thm} \label{thm:system}
The following recurrence relations hold: 
\begin{equation} \label{eq:system}
\left\{ \begin{alignedat}{2}
& \FF_{0} & &{}= 1, \\ 
& \sum_{l = 0}^{n-k} (-1)^{l} (\FH_{n-l-k}^{k+1} - \FH_{n-l-k-2}^{k+1}) \FF_{l} & & {}= 0 \quad \text{for $0 \le k \le n-1$}. 
\end{alignedat}
\right. 
\end{equation}
\end{thm}

First, by applying the Demazure operator $\SD_{1}$ to \eqref{eq:initial}, 
we obtain the following. 

\begin{prop} \label{prop:system_1}
We have 
\begin{equation} \label{eq:secondary}
\sum_{l = 0}^{n-1} (-1)^{l} \e^{-(n-l)\ve_{1}} \left( \sum_{r = 0}^{n-l-1} \e^{r(\ve_{1}-\ve_{2})} \right) \left( \sum_{s = 0}^{n-l-1} \e^{s(\ve_{1}+\ve_{2})} \right) \FF_{l} = 0. 
\end{equation}
\end{prop}
\begin{proof}
By multiplying both sides of \eqref{eq:initial} by $\e^{\ve_{1}}$, we obtain
\begin{equation}
\sum_{l = 0}^{n-1} (-1)^{l} (\e^{-(n-l-1)\ve_{1}} + \e^{(n-l+1)\ve_{1}}) \FF_{l} + (-1)^{n}\e^{\ve_{1}}\FF_{n} = 0. 
\end{equation}
Then, by applying $\SD_{1}$, we see that 
\begin{align}
0 &= \SD_{1}\left(\sum_{l = 0}^{n-1} (-1)^{l} (\e^{-(n-l-1)\ve_{1}} + \e^{(n-l+1)\ve_{1}}) \FF_{l} + (-1)^{n}e^{\ve_{1}}\FF_{n}\right) \\ 
&= \sum_{l = 0}^{n-1} (-1)^{l} (D_{1}(\e^{-(n-l-1)\ve_{1}}) + D_{1}(\e^{(n-l+1)\ve_{1}})) \FF_{l} + (-1)^{n}D_{1}(\e^{\ve_{1}})\FF_{n} \\ 
\begin{split}
&= \sum_{l = 0}^{n-1} (-1)^{l} \left( \frac{\e^{-(n-l-1)\ve_{1}}-\e^{\alpha_{1}}\e^{s_{1}(-(n-l-1)\ve_{1})}}{1-\e^{\alpha_{1}}} + \frac{\e^{(n-l+1)\ve_{1}}-\e^{\alpha_{1}}\e^{s_{1}((n-l+1)\ve_{1})}}{1-\e^{\alpha_{1}}} \right) \FF_{l} \\ 
& \hspace{100mm}  + (-1)^{n}\frac{\e^{\ve_{1}} - \e^{\alpha_{1}}\e^{s_{1}\ve_{1}}}{1-\e^{\alpha_{1}}}\FF_{n} 
\end{split} \\ 
\begin{split}
&= \sum_{l = 0}^{n-1} (-1)^{l} \left( \frac{\e^{-(n-l-1)\ve_{1}}-\e^{\ve_{1}-\ve_{2}}\e^{-(n-l-1)\ve_{2}} + \e^{(n-l+1)\ve_{1}}-\e^{\ve_{1}-\ve_{2}}\e^{(n-l+1)\ve_{2}}}{1-\e^{\alpha_{1}}} \right) \FF_{l} \\ 
& \hspace{100mm}  + (-1)^{n}\frac{\e^{\ve_{1}} - \e^{\ve_{1}-\ve_{2}}\e^{\ve_{2}}}{1-\e^{\alpha_{1}}}\FF_{n} 
\end{split} \\ 
&= \e^{\ve_{1}} \sum_{l = 0}^{n-1} (-1)^{l} \frac{\e^{-(n-l)\ve_{1}}-\e^{-(n-l)\ve_{2}} + \e^{(n-l)\ve_{1}}-\e^{(n-l)\ve_{2}}}{1-\e^{\alpha_{1}}} \FF_{l} \\ 
&= \e^{\ve_{1}} \sum_{l = 0}^{n-1} (-1)^{l} \frac{(\e^{-(n-l)\ve_{1}}-\e^{-(n-l)\ve_{2}})(1-\e^{(n-l)(\ve_{1}+\ve_{2})})}{1-\e^{\alpha_{1}}} \FF_{l} \\ 
&= \e^{\ve_{1}}(1-\e^{\ve_{1}+\ve_{2}}) \sum_{l = 0}^{n-1} (-1)^{l} \left( \sum_{s = 0}^{n-l-1} \e^{s(\ve_{1}+\ve_{2})} \right) \frac{\e^{-(n-l)\ve_{1}}-\e^{-(n-l)\ve_{2}}}{1-\e^{\alpha_{1}}} \FF_{l} \\ 
&= \e^{\ve_{1}}(1-\e^{\ve_{1}+\ve_{2}}) \sum_{l = 0}^{n-1} (-1)^{l} \left( \sum_{s = 0}^{n-l-1} \e^{s(\ve_{1}+\ve_{2})} \right) \frac{\e^{-(n-l)\ve_{1}}(1 -\e^{(n-l)(\ve_{1}-\ve_{2})})}{1-\e^{\ve_{1}-\ve_{2}}} \FF_{l} \\ 
&= \e^{\ve_{1}}(1-\e^{\ve_{1}+\ve_{2}}) \sum_{l = 0}^{n-1} (-1)^{l} \e^{-(n-l)\ve_{1}} \left( \sum_{r=0}^{n-l-1} \e^{r(\ve_{1}-\ve_{2})} \right) \left( \sum_{s = 0}^{n-l-1} \e^{s(\ve_{1}+\ve_{2})} \right) \FF_{l}.  
\end{align}
Finally, by dividing both sides of this equation by $\e^{\ve_{1}}(1-\e^{\ve_{1}+\ve_{2}})$, we obtain 
\begin{equation}
\sum_{l = 0}^{n-1} (-1)^{l} \e^{-(n-l)\ve_{1}} \left( \sum_{r=0}^{n-l-1} \e^{r(\ve_{1}-\ve_{2})} \right) \left( \sum_{s = 0}^{n-l-1} \e^{s(\ve_{1}+\ve_{2})} \right) \FF_{l} = 0, 
\end{equation}
as desired. 
This proves the proposition. 
\end{proof}

Then, by successively applying Demazure operators to \eqref{eq:secondary}, we obtain the other relations. 
For this purpose, we use the following easy lemma. 

\begin{lem} \label{lem:exponential_identity}
\begin{enu}
\item For $k, l \ge 0$ with $k+l\ge 1$ and $1 \le p < q \le n$, we have 
\begin{equation}
\e^{-k\ve_{p}}\e^{-l\ve_{q}} - \e^{l\ve_{p}}\e^{k\ve_{q}} = \e^{-k\ve_{p}}\e^{-l\ve_{q}} (1 - \e^{\ve_{p}+\ve_{q}}) \sum_{t = 0}^{k+l-1} \e^{t(\ve_{p}+\ve_{q})}. 
\end{equation}
\item For $k \ge 0$, we have 
\begin{equation}
\begin{split}
&\e^{-k\ve_{m}} \sum_{s = 0}^{k-1} \e^{s(\ve_{m}-\ve_{m+1})} - \e^{k\ve_{m}} \sum_{s = 1}^{k} \e^{s(-\ve_{m}+\ve_{m+1})} \\ 
& = \e^{-k\ve_{m}} (1-\e^{\ve_{m}+\ve_{m+1}}) \sum_{t = 0}^{k-1} \e^{t(\ve_{m}-\ve_{m+1})} \sum_{t = 0}^{k-1} \e^{t(\ve_{m}+\ve_{m+1})}. 
\end{split}
\end{equation}
\end{enu}
\end{lem}

\begin{prop} \label{prop:system_2}
For $2 \le k \le n-1$, we have 
\begin{equation} \label{eq:system_arbitrary}
\begin{split}
& \sum_{l = 0}^{n-k} (-1)^{l} \left( \sum_{r_{1}=k-1}^{n-l-1} \sum_{s_{1}=0}^{n-l-1-r_{1}} \sum_{r_{2}=k-2}^{r_{1}-1} \sum_{s_{2}=0}^{r_{1}-1-r_{2}} \cdots \sum_{r_{k-1}=1}^{r_{k-2}-1} \sum_{s_{k-1}=0}^{r_{k-2}-1-r_{k-1}} \right. \\ 
& \hspace{30mm} \e^{(l+r_{1}+2s_{1})\ve_{1}+(-r_{1}+r_{2}+2s_{2})\ve_{2}+\cdots+(-r_{k-2}+r_{k-1}+2s_{k-1})\ve_{k-1}+(-r_{k-1})\ve_{k}} \\ 
& \left. \hspace{40mm} \times \sum_{p_{k}=0}^{r_{k-1}-1}\e^{p_{k}(\ve_{k}-\ve_{k+1})} \sum_{q_{k}=0}^{r_{k-1}-1} \e^{q_{k}(\ve_{k}+\ve_{k+1})} \right) \FF_{l} = 0. 
\end{split}
\end{equation}
\end{prop}
\begin{proof}
We prove the proposition by induction on $k$. 
First, we consider the case $k = 2$; in this case, $n \ge 3$. By multiplying both sides of \eqref{eq:secondary} by $\e^{(n-1)\ve_{1}}$, we obtain 
\begin{equation} \label{eq:key_identity_induction_k=2-1}
\sum_{l = 0}^{n-1} (-1)^{l} \e^{l\ve_{1}} \left( \sum_{r = 0}^{n-l-1} \sum_{s = 0}^{n-l-1} \e^{r(\ve_{1}-\ve_{2}) + s(\ve_{1}+\ve_{2})} \right) \FF_{l} = 0. 
\end{equation}
Then, by multiplying both sides of \eqref{eq:key_identity_induction_k=2-1} by $\e^{\ve_{2}}$, we obtain 
\begin{equation} \label{eq:key_identity_induction_k=2-2}
\sum_{l = 0}^{n-1} (-1)^{l} \left( \sum_{r = 0}^{n-l-1} \sum_{s = 0}^{n-l-1} \e^{(l+r+s)\ve_{1} + (-r+s+1)\ve_{2}} \right) \FF_{l} = 0. 
\end{equation}
Here note that 
$\pair{(l+r+s)\ve_{1} + (-r+s+1)\ve_{2}}{\alpha_{2}^{\vee}} = -r+s+1.$ 
Hence, by applying the Demazure operator $\SD_{2}$ to both sides of \eqref{eq:key_identity_induction_k=2-2}, we obtain 
\begin{equation} \label{eq:key_identity_induction_k=2-3}
\begin{split}
\sum_{l = 0}^{n-2} (-1)^{l} & \left( \underbrace{\sum_{r = 1}^{n-l-1} \sum_{s = 0}^{r-1} \e^{(l+r+s)\ve_{1} + (-r+s+1)\ve_{2}} (1 + \e^{\alpha_{2}} + \cdots + \e^{(r-s-1)\alpha_{2}})}_{(*)} \right. \\ 
& \left. -\underbrace{\sum_{r = 0}^{n-l-2} \sum_{s = r+1}^{n-l-1} \e^{(l+r+s)\ve_{1} + (-r+s+1)\ve_{2}} (\e^{-\alpha_{2}} + \e^{-2\alpha_{2}} + \cdots + \e^{-(-r+s)\alpha_{2}} )}_{(**)} \right) \FF_{l} = 0. 
\end{split}
\end{equation}
In $(*)$, we put 
\begin{equation} \label{eq:conv_1}
r_{1} = r-s, \quad s_{1} = s,
\end{equation}
while in $(**)$, we put 
\begin{equation} \label{eq:conv_2}
r_{1} = s-r, \quad s_{1} = r.
\end{equation} 
Then we see that 
\begin{align}
& \text{(LHS) of \eqref{eq:key_identity_induction_k=2-3}} \\ 
\begin{split}
&= \sum_{l = 0}^{n-2} (-1)^{l} \e^{l\ve_{1}} \left( \sum_{r_{1} = 1}^{n-l-1} \sum_{s_{1} = 0}^{n-l-1-r_{1}} \e^{(r_{1}+2s_{1})\ve_{1} + (-r_{1}+1)\ve_{2}} (1 + \e^{\alpha_{2}} + \cdots + \e^{(r_{1}-1)\alpha_{2}}) \right. \\ 
& \hspace{30mm} \left. - \sum_{r_{1}=1}^{n-l-1} \sum_{s_{1}=0}^{n-l-1-r_{1}} \e^{(r_{1}+2s_{1})\ve_{1} + (r_{1}+1)\ve_{2}} (\e^{-\alpha_{2}} + \e^{-2\alpha_{2}} + \cdots + \e^{-r_{1}\alpha_{1}}) \right) \FF_{l}
\end{split} \\ 
&= \sum_{l=0}^{n-2} (-1)^{l} \e^{l\ve_{1}+\ve_{2}} \left( \sum_{r_{1}=1}^{n-l-1} \sum_{s_{1}=0}^{n-l-1-r_{1}} \e^{(r_{1}+2s_{1})\ve_{1}} \right. \\ 
& \hspace{40mm} \left. \times \left( \e^{-r_{1}\ve_{2}} \sum_{t=0}^{r_{1}-1} \e^{t(\ve_{2}-\ve_{3})} - \e^{r_{1}\ve_{2}} \sum_{t=1}^{r_{1}} \e^{t(-\ve_{2}+\ve_{3})} \right) \right) \FF_{l} \\ 
\begin{split}
&= \sum_{l=0}^{n-2} (-1)^{l} \e^{l\ve_{1}+\ve_{2}} \left( \sum_{r_{1}=1}^{n-l-1} \sum_{s_{1}=0}^{n-l-1-r_{1}} \e^{(r_{1}+2s_{1})\ve_{1}} \right. \\ 
& \hspace{40mm} \left. \times \left( \e^{-r_{1}\ve_{2}} (1-\e^{\ve_{2}+\ve_{3}}) \sum_{p_{2}=0}^{r_{1}-1}\e^{p_{2}\ve_{2}-\ve_{3})} \sum_{q_{2}=0}^{r_{1}-1} \e^{q_{2}(\ve_{2}+\ve_{3})} \right) \right) \FF_{l}, 
\end{split} \label{eq:key_identity_induction_k=2-4}
\end{align}
where, for the last equality, we have used Lemma~\ref{lem:exponential_identity}\,(2). 
Hence, by dividing the rightmost-hand side of \eqref{eq:key_identity_induction_k=2-4} by $\e^{\ve_{2}}(1-\e^{\ve_{2}+\ve_{3}})$, we obtain 
\begin{align}
\sum_{l=0}^{n-2} (-1)^{l} \left( \sum_{r_{1}=1}^{n-l-1} \sum_{s_{1}=0}^{n-l-1-r_{1}} \e^{(l+r_{1}+2s_{1})\ve_{1}} \left( \e^{-r_{1}\ve_{2}} \sum_{p_{2}=0}^{r_{1}-1}\e^{p_{2}\ve_{2}-\ve_{3})} \sum_{q_{2}=0}^{r_{1}-1} \e^{q_{2}(\ve_{2}+\ve_{3})} \right) \right) \FF_{l} = 0. 
\end{align}
This proves the proposition for $k = 2$. 

Now, let $3 \le k \le n-1$. 
We assume that the proposition holds for $k-1$, and 
prove the assertion of the proposition for $k$; 
note that $n \ge k+1$ in this case. 
By the induction hypothesis, we have 
\begin{equation}
\begin{split}
& \sum_{l = 0}^{n-k+1} (-1)^{l} \Biggl( \sum_{r_{1}=k-2}^{n-l-1} \sum_{s_{1}=0}^{n-l-1-r_{1}} \sum_{r_{2}=k-3}^{r_{1}-1} \sum_{s_{2}=0}^{r_{1}-1-r_{2}} \cdots \sum_{r_{k-2}=1}^{r_{k-3}-1} \sum_{s_{k-2}=0}^{r_{k-3}-1-r_{k-2}} \\ 
& \hspace{30mm} \e^{(l+r_{1}+2s_{1})\ve_{1}+(-r_{1}+r_{2}+2s_{2})\ve_{2}+\cdots+(-r_{k-3}+r_{k-2}+2s_{k-2})\ve_{k-2}+(-r_{k-2})\ve_{k-1}} \\ 
& \hspace{40mm} \times \sum_{p_{k-1}=0}^{r_{k-2}-1}\e^{p_{k-1}(\ve_{k-1}-\ve_{k})} \sum_{q_{k-1}=0}^{r_{k-2}-1} \e^{q_{k-1}(\ve_{k-1}+\ve_{k})} \Biggr) \FF_{l} = 0. 
\end{split}
\end{equation}
By
\begin{enu}
\item multiplying both sides of this equation by $\e^{\ve_{k}}$, 
\item applying the Demazure operator $\SD_{k}$ to both sides, 
\item making a change of variables similar to \eqref{eq:conv_1} and \eqref{eq:conv_2}, and 
\item dividing both sides by $\e^{\ve_{k}}(1-\e^{\ve_{k}+\ve_{k+1}})$, 
\end{enu}
we deduce that 
\begin{equation}
\begin{split}
& \sum_{l = 0}^{n-k+1} (-1)^{l} \Biggl( \sum_{r_{1}=k-1}^{n-l-1} \sum_{s_{1}=0}^{n-l-1-r_{1}} \sum_{r_{2}=k-2}^{r_{1}-1} \sum_{s_{2}=0}^{r_{1}-1-r_{2}} \cdots \sum_{r_{k-1}=1}^{r_{k-2}-1} \sum_{s_{k-1}=0}^{r_{k-2}-1-r_{k-1}} \\ 
& \quad \e^{(l+r_{1}+2s_{1})\ve_{1}+(-r_{1}+r_{2}+2s_{2})\ve_{2}+\cdots+(-r_{k-3}+r_{k-2}+2s_{k-2})\ve_{k-2}+(-r_{k-2}+r_{k-1}+2s_{k-1})\ve_{k-1}+(-r_{k-1})\ve_{k}} \\ 
& \hspace{70mm} \times \sum_{p_{k} = 0}^{r_{k-1}-1} \e^{p_{k}(\ve_{k}-\ve_{k+1})} \sum_{q_{k} = 0}^{r_{k-1}-1} \e^{q_{k}(\ve_{k}+\ve_{k+1})} \Biggr) \FF_{l} = 0, 
\end{split}
\end{equation}
as desired. Thus, by induction on $k$, the proposition is proved. 
\end{proof}

Also, we can prove the following proposition for complete symmetric polynomials. Since the proof of this proposition is elementary, we leave it to the reader.  
\begin{prop} \label{prop:complete_symmetric}
In the Laurent polynomial ring $\BZ[x_{1}^{\pm 1}, \ldots, x_{n}^{\pm 1}]$, the following hold. 
\begin{enu}
\item We have $h_{0}(x_{1}, \ldots, x_{n}, x_{n}^{-1}, \ldots, x_{1}^{-1}) = 1$. 

\item For $m \ge 1$, we have 
\begin{equation}
x_{1}^{m} + x_{1}^{-m} = \begin{cases} h_{1}(x_{1}, x_{1}^{-1}) & \text{if $m = 1$}, \\ h_{m}(x_{1}, x_{1}^{-1}) - h_{m-2}(x_{1}, x_{1}^{-1}) & \text{if $m \ge 2$}. \end{cases}
\end{equation}

\item For $m \ge 1$, we have 
\begin{align}
& x_{1}^{-m} \left( \sum_{k = 0}^{m} (x_{1}x_{2}^{-1})^{k} \right) \left( \sum_{l = 0}^{m} (x_{1}x_{2})^{l} \right) \\ 
&= \begin{cases} h_{1}(x_{1}, x_{2}, x_{2}^{-1}, x_{1}^{-1}) & \text{if $m = 1$}, \\ h_{m}(x_{1}, x_{2}, x_{2}^{-1}, x_{1}^{-1}) - h_{m-2}(x_{1}, x_{2}, x_{2}^{-1}, x_{1}^{-1}) & \text{if $m \ge 2$}. \end{cases}
\end{align}

\item If $n \ge 3$, then for $m \ge 1$, we have 
\begin{align}
& \sum_{r_{1} = n-2}^{m+n-2} \sum_{s_{1}=0}^{m+n-2-r_{1}} \sum_{r_{2}=n-3}^{r_{1}-1} \sum_{s_{2}=0}^{r_{1}-1-r_{2}} \cdots \sum_{r_{n-2}=1}^{r_{n-3}-1} \sum_{s_{n-2}=0}^{r_{n-3}-1-r_{n-2}} \\ 
& \hspace{20mm} x_{1}^{-(m+n-2)+r_{1}+2s_{1}} x_{2}^{-r_{1}+r_{2}+2s_{2}+1} \cdots x_{n-2}^{-r_{n-3}+r_{n-2}+2s_{n-2}+1} x_{n-1}^{-r_{n-2}} \\ 
& \hspace{30mm} \times \left( \sum_{k = 0}^{r_{n-2}-1} (x_{n-1}x_{n}^{-1})^{k} \right) \left( \sum_{l = 0}^{r_{n-2}-1} (x_{n-1}x_{n})^l \right) \\ 
&= \begin{cases} h_{1}(x_{1}, \ldots, x_{n}, x_{n}^{-1}, \ldots, x_{1}^{-1}) & \text{if $m = 1$}, \\ h_{m}(x_{1}, \ldots, x_{n}, x_{n}^{-1}, \ldots, x_{1}^{-1}) - h_{m-2}(x_{1}, \ldots, x_{n}, x_{n}^{-1}, \ldots, x_{1}^{-1}) & \text{if $m \ge 2$}. \end{cases} 
\end{align}
\end{enu}
\end{prop}

\begin{proof}[Proof of Theorem~\ref{thm:system}]
By multiplying both sides of \eqref{eq:secondary} by $\e^{\ve_{1}}$ and applying Proposition~\ref{prop:complete_symmetric}\,(3) with $x_{1} = \e^{\ve_{1}}$, $x_{2} = \e^{\ve_{2}}$, and $m = n-l-1$, we obtain 
\begin{equation}
\sum_{l = 0}^{n-1} (-1)^{l} (\FH_{n-l-1}^{2} - \FH_{n-l-3}^{2}) \FF_{l} = 0. 
\end{equation}
For $2 \le k \le n-1$, by multiplying both sides of \eqref{eq:system_arbitrary} by $\e^{-(2n-k-2)\ve_{1} + \ve_{2} + \cdots + \ve_{n-1}}$ and applying Proposition~\ref{prop:complete_symmetric}\,(4) with $x_{1} = \e^{\ve_{1}}, \ldots, x_{k+1} = \e^{\ve_{k+1}}$, and $m = n-l-k$, we obtain 
\begin{equation}
\sum_{l = 0}^{n-k} (-1)^{l} (\FH_{n-l-k}^{k+1} - \FH_{n-l-k-2}^{k+1}) \FF_{l} = 0. 
\end{equation}
This proves Theorem~\ref{thm:system}. 
\end{proof}

\subsection{The solution to the recurrence relations}
In this subsection, we solve the recurrence relations \eqref{eq:system} for $\FF_{l}$, $0 \le l \le n$, with coefficients in $R(T)$. 
For $0 \le l \le 2n$, we define $\FE_{l}^{n} \in R(T) \subset \KTQG$ by 
\begin{equation}
\FE_{l}^{n} := e_{l}(\e^{\ve_{1}}, \ldots, \e^{\ve_{n-1}}, \e^{\ve_{n}}, \e^{-\ve_{n}}, \e^{-\ve_{n-1}}, \ldots, \e^{-\ve_{1}}). 
\end{equation}
Observe that $\FE_{n+l}^{n} = \FE_{n-l}^{n}$ for $1 \le l \le n$ since 
\begin{equation} \label{eq:symmetry_of_elementary_symmetric}
e_{n+l}(x_{1}, \ldots, x_{n}, x_{n}^{-1}, \ldots, x_{1}^{-1}) = e_{n-l}(x_{1}, \ldots, x_{n}, x_{n}^{-1}, \ldots, x_{1}^{-1}) 
\end{equation} 
for $1 \le l \le n$. 
\begin{thm} \label{thm:relation_semi-infinite}
For $0 \le l \le 2n$, we have $\FF_{l} = \FE_{l}^{n}$. 
\end{thm}
\begin{proof}
By Proposition~\ref{prop:symmetry}, it suffices to prove the theorem for $0 \le l \le n$. For this purpose, we will prove the following: 
\begin{align}
\FE_{0}^{n} &= 1, \label{eq:system_1} \\ 
\sum_{l = 0}^{n-k} (-1)^{l} (\FH_{n-l-k}^{k+1} - \FH_{n-l-k-2}^{k+1}) \FE_{l}^{n} &= 0, \quad 0 \le k \le n-1. \label{eq:system_2}
\end{align}
It follows from these equations that $(\FE_{0}^{n}, \ldots, \FE_{n}^{n})$ is a solution of the recurrence relations \eqref{eq:system}. 
Since the elements $\FF_{l}$, $0 \le l \le n$, are uniquely determined by the recurrence relations \eqref{eq:system}, we can conclude that $\FF_{l} = \FE_{l}^{n}$ for $0 \le l \le n$, as desired. 

First, equation \eqref{eq:system_1} is obvious from the definition of elementary symmetric polynomials. 
We will prove \eqref{eq:system_2}. 
Since 
\begin{equation}
\sum_{l = 0}^{\infty} h_{l}(x_{1}, \ldots, x_{d}) t^{l} = \prod_{i = 1}^{d} \frac{1}{1-x_{i}t} 
\end{equation}
for $1 \le k+1 \le n$, we have 
\begin{equation} \label{eq:sum_complete_symmetric}
\sum_{l = 0}^{\infty} \FH_{l}^{k+1} t^{l} = \left( \prod_{i = 1}^{k+1} \frac{1}{1 - \e^{\ve_{i}}t} \right) \left( \prod_{j = 1}^{k+1} \frac{1}{1 - \e^{-\ve_{j}}t} \right). 
\end{equation}
By multiplying both sides of \eqref{eq:sum_complete_symmetric} by $t^{2}$ and then subtracting the resulting equation from \eqref{eq:sum_complete_symmetric}, we obtain 
\begin{equation} \label{eq:generating_function_1}
\sum_{l = 0}^{\infty} (\FH_{l}^{k+1} - \FH_{l-2}^{k+1}) t^{l} = (1-t^{2}) \left( \prod_{i = 1}^{k+1} \frac{1}{1 - \e^{\ve_{i}}t} \right) \left( \prod_{j = 1}^{k+1} \frac{1}{1 - \e^{-\ve_{j}}t} \right). 
\end{equation}
Also, since 
\begin{equation}
\sum_{l = 0}^{d} e_{l}(x_{1}, \ldots, x_{d}) t^{l} = \prod_{i = 1}^{d} (1 + x_{i}t), 
\end{equation}
it follows that 
\begin{equation} \label{eq:generating_function_2}
\sum_{l = 0}^{2n} \FE_{l}^{n} t^{d} = \left( \prod_{i = 1}^{n} (1 + \e^{\ve_{i}}t) \right) \left( \prod_{j = 1}^{n} (1 + \e^{-\ve_{j}}t) \right). 
\end{equation}
By multiplying \eqref{eq:generating_function_2} and the equation obtained from \eqref{eq:generating_function_1} by replacing $t$ with $-t$, 
we obtain 
\begin{equation} \label{eq:generating_function_3}
\begin{split}
& \left( \sum_{l = 0}^{\infty} (-1)^{l} (\FH_{l}^{k+1} - \FH_{l-2}^{k+1}) t^{l} \right) \left( \sum_{m = 0}^{2n} \FE_{m}^{n} t^{m} \right) \\ 
& = (1-t^{2}) \left( \prod_{i = k+2}^{n} (1 + \e^{\ve_{i}}t) \right) \left( \prod_{j = k+2}^{n} (1 + \e^{-\ve_{j}}t) \right). 
\end{split}
\end{equation}
The left-hand side of \eqref{eq:generating_function_3} can be rewritten as: 
\begin{align}
& \left( \sum_{l = 0}^{\infty} (-1)^{l} (\FH_{l}^{k+1} - \FH_{l-2}^{k+1}) t^{l} \right) \left( \sum_{m = 0}^{2n} \FE_{m}^{n} t^{m} \right) \nonumber \\ 
&= \sum_{l = 0}^{\infty} \left( \sum_{d = 0}^{l} (-1)^{l-d} (\FH_{l-d}^{k+1} - \FH_{l-d-2}^{k+1}) \FE_{d}^{n} \right) t^{l}.  \label{eq:generating_function_3_LHS}
\end{align}
By setting $e_{m}(x_{1}, \ldots, x_{d}) = 0$ if $m < 0$ or $m > d$, the right-hand side of \eqref{eq:generating_function_3} can be rewritten as: 
\begin{align}
& (1-t^{2}) \left( \prod_{i = k+2}^{n} (1 + \e^{\ve_{i}}t) \right) \left( \prod_{j = k+2}^{n} (1 + \e^{-\ve_{j}}t) \right) \nonumber \\ 
&= (1-t^{2}) \sum_{m = 0}^{2(n-k-1)} e_{m}(\e^{\ve_{k+2}}, \ldots, \e^{\ve_{n}}, \e^{-\ve_{n}}, \ldots, \e^{-\ve_{k+2}}) t^{m} \nonumber \\ 
&= \sum_{m = 0}^{2(n-k-1)} (e_{m}(\e^{\ve_{k+2}}, \ldots, \e^{\ve_{n}}, \e^{-\ve_{n}}, \ldots, \e^{-\ve_{k+2}}) - e_{m-2}(\e^{\ve_{k+2}}, \ldots, \e^{\ve_{n}}, \e^{-\ve_{n}}, \ldots, \e^{-\ve_{k+2}})) t^{m}. \label{eq:generating_function_3_RHS}
\end{align}
Therefore, by comparing the coefficients of $t^{n-k}$ in \eqref{eq:generating_function_3_LHS} and \eqref{eq:generating_function_3_RHS}, we see that 
\begin{align}
& \sum_{d = 0}^{n-k} (-1)^{n-k-d} (\FH_{n-k-d}^{k+1} - \FH_{n-k-d-2}^{k+1}) \FE_{d}^{n} \\ 
&= e_{n-k}(\e^{\ve_{k+2}}, \ldots, \e^{\ve_{n}}, \e^{-\ve_{n}}, \ldots, \e^{-\ve_{k+2}}) - e_{n-k-2}(\e^{\ve_{k+2}}, \ldots, \e^{\ve_{n}}, \e^{-\ve_{n}}, \ldots, \e^{-\ve_{k+2}}) \\ 
&= e_{(n-k-1)+1}(\underbrace{\e^{\ve_{k+2}}, \ldots, \e^{\ve_{n}}, \e^{-\ve_{n}}, \ldots, \e^{-\ve_{k+2}}}_{\text{$2(n-k-1)$ variables}}) - e_{(n-k-1)-1}(\e^{\ve_{k+2}}, \ldots, \e^{\ve_{n}}, \e^{-\ve_{n}}, \ldots, \e^{-\ve_{k+2}}) \\ 
&= 0 \quad \text{by \eqref{eq:symmetry_of_elementary_symmetric}}. 
\end{align}
Thus, we deduce that 
\begin{equation}
\sum_{l = 0}^{n-k} (-1)^{l} (\FH_{n-l-k}^{k+1} - \FH_{n-l-k-2}^{k+1}) \FE_{l}^{n} = (-1)^{n-k} \sum_{l = 0}^{n-k} (-1)^{n-k-l} (\FH_{n-l-k}^{k+1} - \FH_{n-l-k-2}^{k+1}) \FE_{l}^{n} = 0. 
\end{equation}
This proves the theorem. 
\end{proof}

\section{Borel-type presentation} \label{sec:proof}
The aim of this section is to give a proof of Theorem~\ref{thm:main}. 
We assume that $G = \Sp$, the symplectic group of rank $n$. 
First, we derive some relations in $QK_{T}(G/B)$ from the corresponding ones in $\KTQG$, given in Theorem~\ref{thm:relation_semi-infinite}. 
Then, based on these relations, we prove the existence of a homomorphism $(R(T)\bra{Q})[z_{1}^{\pm 1}, \ldots, z_{n}^{\pm 1}] \rightarrow QK_{T}(G/B)$ of $R(T)\bra{Q}$-algebras which annihilates an ideal $\CI^{Q}$ of $(R(T)\bra{Q})[z_{1}^{\pm 1}, \ldots, z_{n}^{\pm 1}]$. 
Finally, we prove that the induced homomorphism $(R(T)\bra{Q})[z_{1}^{\pm 1}, \ldots, z_{n}^{\pm 1}]/\CI^{Q} \rightarrow QK_{T}(G/B)$ 
is, in fact, an isomorphism. 

\subsection{\texorpdfstring{Some relations in $QK_{T}(G/B)$}{Some relations in QKT(G/B)}}
We derive some relations in $QK_{T}(G/B)$ from the corresponding ones in $\KTQG$, given in Theorem~\ref{thm:relation_semi-infinite}). 

\begin{defn}
Let $I \subset [1, \overline{1}]$. 
\begin{enu}
\item For $1 \le j \le n$, we define $\varphi_{I}^{Q}(j) \in \BZ\bra{Q}$ by 
\begin{equation}
\vp_{I}^{Q}(j) := \begin{cases}
\dfrac{1}{1-Q_{j}} & \text{if $j, j+1 \in I$}, \\ 
1 & \text{otherwise}. 
\end{cases}
\end{equation}

\item For $2 \le j \le n$, we define $\varphi_{I}^{Q}(\overline{j}) \in \BZ\bra{Q}$ by 
\begin{equation}
\vp_{I}^{Q}(\overline{j}) := \begin{cases}
1 + \dfrac{Q_{j-1} Q_{j} \cdots Q_{n}}{1-Q_{j-1}} & \text{if $I = \{ \cdots < j-1 < \overline{j-1} < \cdots\}$}, \\ 
\dfrac{1}{1-Q_{j-1}} & \text{if $\overline{j}, \overline{j-1} \in I$}, \\ 
1 & \text{othewise}. 
\end{cases}
\end{equation}

\item We set $\vp_{I}^{Q}(\overline{1}) := 1 \in \BZ\bra{Q}$. 
\end{enu}
\end{defn}

\begin{defn}
For $0 \le l \le 2n$, we define $\CF_{l} \in QK_{T}(G/B)$ by 
\begin{equation}
\CF_{l} := \sum_{\substack{I \subset [1, \overline{1}] \\ |I| = l}} \left( \prod_{1 \le j \le \overline{1}} \vp_{I}^{Q}(j) \right) \left( {\prod_{j \in I}}^{\star} [\CO_{G/B}(-\ve_{j})] \right), 
\end{equation}
where ${\prod}^{\star}$ denotes the product with respect to the quantum product $\star$. 
\end{defn}

\begin{thm} \label{thm:semi-inf->QK}
For $0 \le l \le 2n$, we have $\Phi(\CF_{l}) = \FF_{l}$. 
\end{thm}
\begin{proof}
In this proof, we thought of $1/(1-\st_{i})$ for $1 \le i \le n$ as the infinite sum $\sum_{k = 0}^{\infty} \st_{i}^{k}$; 
this infinite sum is a well-defined operator on $\KTQG$. 
We define $\varphi_{I}^{\sinf}(j)$ for $I \subset [1, \overline{1}]$ and $j \in [1, \overline{1}]$ as follows. 
\begin{enu}
\item For $1 \le j \le n$, we set 
\begin{equation}
\varphi_{I}^{\sinf}(j) := \begin{cases}
\dfrac{1}{1-\st_{j}} & \text{if $j, j+1 \in I$}, \\ 
1 & \text{otherwise}. 
\end{cases}
\end{equation}
\item For $2 \le j \le n$, we set 
\begin{equation}
\vp_{I}^{\sinf}(\overline{j}) := \begin{cases}
1 + \dfrac{\st_{j-1} \st_{j} \cdots \st_{n}}{1-\st_{j-1}} & \text{if $I = \{ \cdots < j-1 < \overline{j-1} < \cdots\}$}, \\ 
\dfrac{1}{1-\st_{j-1}} & \text{if $\overline{j}, \overline{j-1} \in I$}, \\ 
1 & \text{othewise}. 
\end{cases}
\end{equation}
\item Set $\varphi_{I}^{\sinf}(\overline{1}) := 1$. 
\end{enu}
By Proposition~\ref{prop:line_bundle_correspondence}, we see that 
\begin{align}
\Phi(\CF_{l}) &= \sum_{\substack{I \subset [1, \overline{1}] \\ |I| = l}} \left( \prod_{1 \le j \le \overline{1}} \vp_{I}^{\sinf}(j) \right) \left( \prod_{\substack{1 \le j \le n \\ j \in I}} (1 - \st_{j-1}) \right) \left( \prod_{\substack{1 \le j \le n \\ \overline{j} \in I}} (1-\st_{j}) \right) [\CO_{\bQG}(-\ve_{I})] \\ 
&= \sum_{\substack{I \subset [1, \overline{1}] \\ |I| = l}} \left( \prod_{1 \le j \le \overline{1}} \vp_{I}^{\sinf}(j) \right) \left( \prod_{\substack{0 \le j \le n-1 \\ j+1 \in I}} (1 - \st_{j}) \right) \left( \prod_{\substack{2 \le j \le n+1 \\ \overline{j-1} \in I}} (1-\st_{j-1}) \right) [\CO_{\bQG}(-\ve_{I})]. 
\end{align}
Let us take and fix $I \subset [1, \overline{1}]$. For $1 \le j \le n-1$, we set 
\begin{equation}
\theta_{I}^{\sinf}(j) := \begin{cases}
1 - \st_{j} & \text{if $j+1 \in I$,} \\ 
1 & \text{otherwise.}
\end{cases}
\end{equation}
Then, for $1 \le j \le n-1$, it follows that 
\begin{align}
\vp_{I}^{\sinf}(j) \theta_{I}^{\sinf}(j) &= \begin{cases}
\dfrac{1}{1-\st_{j}} \cdot (1-\st_{j}) & \text{if $j \in I$ and $j+1 \in I$,} \\ 
1 \cdot 1 & \text{if $j \in I$ and $j+1 \notin I$,} \\ 
1 \cdot (1-\st_{j}) & \text{if $j \notin I$ and $j+1 \in I$,} \\ 
1 \cdot 1 & \text{if $j \notin I$ and $j+1 \notin I$} 
\end{cases} \\ 
&= \begin{cases}
1 - \st_{j} & \text{if $j \notin I$ and $j+1 \in I$,} \\ 
1 & \text{otherwise}
\end{cases} \\ 
&= \psi_{I}^{\sinf}(j). 
\end{align}
In addition, we set 
\begin{equation}
\theta_{I}^{\sinf}(n) := \begin{cases}
1 - \st_{n} & \text{if $\overline{n} \in I$,} \\ 
1 & \text{otherwise}. 
\end{cases}
\end{equation}
Then it follows that 
\begin{align}
\vp_{I}^{\sinf}(n) \theta_{I}^{\sinf}(n) &= \begin{cases}
\dfrac{1}{1-\st_{n}} \cdot (1-\st_{n}) & \text{if $n \in I$ and $\overline{n} \in I$,} \\ 
1 \cdot 1 & \text{if $n \in I$ and $\overline{n} \notin I$,} \\ 
1 \cdot (1-\st_{n}) & \text{if $n \notin I$ and $\overline{n} \in I$,} \\ 
1 \cdot 1 & \text{if $n \notin I$ and $\overline{n} \notin I$}
\end{cases} \\ 
&= \begin{cases}
1-\st_{n} & \text{if $n \notin I$ and $\overline{n} \in I$,} \\ 
1 & \text{otherwise}
\end{cases} \\ 
&= \psi_{I}^{\sinf}(n). 
\end{align}
Also, for $2 \le j \le n$, we set 
\begin{equation}
\theta_{I}^{\sinf}(\overline{j}) := \begin{cases}
1 - \st_{j-1} & \text{if $\overline{j-1} \in I$,} \\ 
1 & \text{otherwise.}
\end{cases}
\end{equation}
Then, for $2 \le j \le n$, it follows that 
\begin{align}
\vp_{I}^{\sinf}(\overline{j}) \theta_{I}^{\sinf}(\overline{j}) &= \begin{cases}
\left( 1 + \dfrac{\st_{j-1} \st_{j} \cdots \st_{n}}{1 - \st_{j-1}} \right) \cdot (1 - \st_{j-1}) & \text{if $I = \{ \cdots < j-1 < \overline{j-1} < \cdots \}$,} \\ 
1 \cdot (1-\st_{j-1}) & \text{\parbox{15em}{if $\overline{j} \notin I$, $\overline{j-1} \in I$, and $I$ is not of the above form,}} \\ 
1 \cdot 1 & \text{if $\overline{j} \notin I$ and $\overline{j-1} \notin I$,} \\
\dfrac{1}{1-\st_{j-1}} \cdot (1-\st_{j-1}) & \text{if $\overline{j} \in I$ and $\overline{j-1} \in I$,} \\ 
1 \cdot 1 & \text{if $\overline{j} \in I$ and $\overline{j-1} \notin I$}
\end{cases} \\ 
&= \begin{cases}
1 - \st_{j-1} + \st_{j-1}\st_{j} \cdots \st_{n} & \text{if $I = \{ \cdots < j-1 < \overline{j-1} < \cdots \}$,} \\ 
1 - \st_{j-1} & \text{\parbox{15em}{if $\overline{j} \notin I$, $\overline{j-1} \in I$, and $I$ is not of the above form,}} \\ 
1 & \text{otherwise}
\end{cases} \\ 
&= \psi_{I}^{\sinf}(\overline{j}). 
\end{align}
Therefore, by setting $\theta_{I}^{\sinf}(\overline{1}) := 1$, we deduce that 
\begin{align}
\Phi(\CF_{l}) &= \sum_{\substack{I \subset [1, \overline{1}] \\ |I| = l}} \left( \prod_{1 \le j \le \overline{1}} \vp_{I}^{\sinf}(j) \theta_{I}^{\sinf}(j) \right) [\CO_{\bQG}(-\ve_{I})] \\ 
&= \sum_{\substack{I \subset [1, \overline{1}] \\ |I| = l}} \left( \prod_{1 \le j \le \overline{1}} \psi_{I}^{\sinf}(j) \right) [\CO_{\bQG}(-\ve_{I})] \\ 
&= \FF_{l}, 
\end{align}
as desired. 
This proves the theorem. 
\end{proof}

For $0 \le l \le 2n$, we define an element $\CE_{l} \in R(T) \subset QK_{T}(G/B)$ by 
\begin{align}
\CE_{l} &:= e_{l}(\e^{-\ve_{1}}, \ldots, \e^{-\ve_{n-1}}, \e^{-\ve_{n}}, \e^{\ve_{n}}, \e^{\ve_{n-1}}, \ldots, \e^{\ve_{1}}) \\ 
&= e_{l}(\e^{\ve_{1}}, \ldots, \e^{\ve_{n-1}}, \e^{\ve_{n}}, \e^{-\ve_{n}}, \e^{-\ve_{n-1}}, \ldots, \e^{-\ve_{1}}). 
\end{align}
\begin{cor} \label{cor:relation_QK}
For $0 \le l \le 2n$, we have $\CF_{l} = \CE_{l}$. 
\end{cor}
\begin{proof}
By Theorems~\ref{thm:semi-inf->QK} and \ref{thm:relation_semi-infinite}, we have
\begin{equation}
\Phi(\CF_{l}-\CE_{l}) = \FF_{l}-\FE_{l}^{n} = 0. 
\end{equation}
Since $\Phi$ is injective, we conclude that $\CF_{l} = \CE_{l}$, as desired. 
\end{proof}

As a consequence of the proof of Theorem~\ref{thm:semi-inf->QK}, we obtain an explicit expression, in terms of line bundles, of the Schubert classes $[\CO^{s_{1}s_{2} \cdots s_{k}}]$ and $[\CO^{s_{1}s_{2} \cdots s_{n}s_{n-1} \cdots s_{k}}]$ for $1 \le k \le n$. 
\begin{defn}
\begin{enu}
\item For $1 \le k \le n$ and $0 \le l \le k$, we define an element $\CF_{l}^{k} \in QK_{T}(G/B)$ by 
\begin{equation}
\CF_{l}^{k} := \sum_{\substack{I \subset [1, k] \\ |I| = l}} \left( \prod_{1 \le j \le \overline{1}} \varphi_{I}^{Q}(j) \right) \left( {\prod_{j \in I}}^{\star} [\CO_{G/B}(-\ve_{j})] \right). 
\end{equation}
\item For $1 \le k \le n$ and $0 \le l \le 2n-k$, we define an element $\CF_{l}^{\overline{k}} \in QK_{T}(G/B)$ by 
\begin{equation}
\CF_{l}^{\overline{k}} := \sum_{\substack{I \subset [1, \overline{k+1}] \\ |I| = l}} \left( \prod_{1 \le j \le \overline{1}} \varphi_{I}^{Q}(j) \right) \left( {\prod_{j \in I}}^{\star} [\CO_{G/B}(-\ve_{j})] \right), 
\end{equation}
where we understand that $\overline{n+1} := n$. 
\end{enu}
\end{defn}

We can prove the following theorem; its proof is almost the same as that of Theorem~\ref{thm:semi-inf->QK}. 
\begin{thm}
\begin{enu}
\item For $1 \le k \le n$ and $0 \le l \le k$, we have $\Phi(\CF_{l}^{k}) = \FF_{l}^{k}$. 
\item For $1 \le k \le n$ and $0 \le l \le 2n-k$, we have $\Phi(\CF_{l}^{\overline{k}}) = \FF_{l}^{\overline{k}}$. 
\end{enu}
\end{thm}
Hence, by Theorem~\ref{thm:solution_recursion}, we obtain the following. 
\begin{cor} \label{cor:explicit_classes}
\begin{enu}
\item For $1 \le k \le n$, the following equality holds in $QK_{T}(G/B)$: 
\begin{equation}
[\CO^{s_{1} \cdots s_{k}}] = \sum_{l=0}^{k} (-1)^{l} \e^{-l\ve_{1}} \CF_{l}^{k}. 
\end{equation}
\item For $1 \le k \le n$, the following equality holds in $QK_{T}(G/B)$: 
\begin{equation}
[\CO^{s_{1} \cdots s_{n} s_{n-1} \cdots s_{k}}] = \sum_{l = 0}^{2n-k} (-1)^{l} \e^{-l\ve_{1}} \CF_{l}^{\overline{k}}. 
\end{equation}
\end{enu}
\end{cor}

\subsection{The existence of a homomorphism}
It follows from Corollary~\ref{cor:relation_QK} that there exists a homomorphism $(R(T)\bra{Q})[z_{1}^{\pm 1}, \ldots, z_{n}^{\pm 1}] \rightarrow QK_{T}(G/B)$ 
of $R(T)\bra{Q}$-algebras, 
which eventually gives a Borel-type presentation of $QK_{T}(G/B)$. 

\begin{defn}
We define a homomorphism $\widehat{\Psi}^{Q}: (R(T)\bra{Q})[z_{1}^{\pm 1}, \ldots, z_{n}^{\pm 1}] \rightarrow QK_{T}(G/B)$ of $R(T)\bra{Q}$-algebras by: 
\begin{equation}
z_{j} \mapsto \frac{1}{1-Q_{j}} [\CO_{G/B}(-\ve_{j})], \quad 1 \le j \le n. 
\end{equation}
\end{defn}

Let us compute the images $\widehat{\Psi}^{Q}(z_{j}^{-1})$, $1 \le j \le n$. By convention, we set $Q_{0} := 0$. 

\begin{lem} \label{lem:inverse_class}
For $1 \le j \le n$, the line bundle class $[\CO_{G/B}(-\ve_{j})] \in QK_{T}(G/B)$ is invertible with respect to the quantum product $\star$ in $QK_{T}(G/B)$, and its inverse is given as: 
\begin{equation}
[\CO_{G/B}(-\ve_{j})]^{-1} = \frac{1}{(1-Q_{j})(1-Q_{j-1})} [\CO_{G/B}(\ve_{j})]. 
\end{equation}
\end{lem}
\begin{proof}
We set 
\begin{align}
\CL &:= \frac{1}{(1-Q_{j})(1-Q_{j-1})} [\CO_{G/B}(\ve_{j})] \star [\CO_{G/B}(-\ve_{j})] \\ 
&= \left( \frac{1}{1-Q_{j}} [\CO_{G/B}(\ve_{j})] \right) \star \left( \left( \frac{1}{1-Q_{j-1}} [\CO_{G/B}(-\ve_{j})] \right) \star [\CO_{G/B}^{e}] \right). 
\end{align}
It suffices to show that $\CL = 1 = [\CO_{G/B}^{e}]$. 
By using the map $\Phi: QK_{T}(G/B) \xrightarrow{\sim} \KTQG$ (see Theorem~\ref{thm:quantum=semi-infinite}), we see that 
\begin{align}
\Phi(\CL) &= \Phi \left( \left( \frac{1}{1-Q_{j}} [\CO_{G/B}(\ve_{j})] \right) \star \left( \left( \frac{1}{1-Q_{j-1}} [\CO_{G/B}(-\ve_{j})] \right) \star [\CO_{G/B}^{e}] \right) \right) \\ 
&= [\CO_{\bQG}(\ve_{j})] \otimes \left( [\CO_{\bQG}(-\ve_{j})] \otimes [\CO_{\bQG(e)}] \right) \quad \text{(by Proposition~\ref{prop:line_bundle_correspondence})} \\ 
&= \left( [\CO_{\bQG}(\ve_{j})] \otimes [\CO_{\bQG}(-\ve_{j})] \right) \otimes [\CO_{\bQG(e)}] \\ 
&= [\CO_{\bQG(e)}]. 
\end{align}
Since $\Phi([\CO_{G/B}^{e}]) = [\CO_{\bQG(e)}]$, we deduce that $\CL = [\CO_{G/B}^{e}] = 1$, as desired.
This proves the lemma. 
\end{proof}

\begin{cor} \label{cor:z_inverse}
For $1 \le j \le n$, we have 
\begin{equation}
\widehat{\Psi}^{Q}(z_{j}^{-1}) = \frac{1}{1-Q_{j-1}} [\CO_{G/B}(\ve_{j})]. 
\end{equation}
\end{cor}
\begin{proof}
We compute as: 
\begin{align}
\widehat{\Psi}^{Q}(z_{j}^{-1}) &= \widehat{\Psi}^{Q}(z_{j})^{-1} \\ 
&= \left( \frac{1}{1-Q_{j}} [\CO_{G/B}(-\ve_{j})] \right)^{-1} \\ 
&= (1-Q_{j}) \cdot [\CO_{G/B}(-\ve_{j})]^{-1} \\ 
&= (1-Q_{j}) \cdot \frac{1}{(1-Q_{j})(1-Q_{j-1})} [\CO_{G/B}(\ve_{j})] \quad \text{(by Lemma~\ref{lem:inverse_class})} \\ 
&= \frac{1}{1-Q_{j-1}} [\CO_{G/B}(\ve_{j})]. 
\end{align}
This proves the corollary. 
\end{proof}

\begin{lem} \label{lem:QK->Borel}
For $0 \le j \le 2n$, we have $\widehat{\Psi}^{Q}(F_{l}) = \CF_{l}$. 
\end{lem}
\begin{proof}
By Corollary~\ref{cor:z_inverse}, we see that 
\begin{align}
& \widehat{\Psi}^{Q}(F_{l}) \\ 
&= \sum_{\substack{I \subset [1, \overline{1}] \\ |I| = l}} \left( \prod_{1 \le j \le \overline{1}} \zeta_{I}^{\poly}(j) \right) \left( {\prod_{\substack{1 \le j \le n \\ j \in I}}}^{\star} \frac{1}{1-Q_{j}} [\CO_{G/B}(-\ve_{j})] \right) \left( {\prod_{\substack{1 \le j \le n \\ \overline{j} \in I}}}^{\star} \frac{1}{1-Q_{j-1}}[\CO_{G/B}(\ve_{j})] \right) \\ 
&= \sum_{\substack{I \subset [1, \overline{1}] \\ |I| = l}} \left( \prod_{1 \le j \le \overline{1}} \zeta_{I}^{\poly}(j) \right) \left( \prod_{\substack{1 \le j \le n \\ j \in I}} \frac{1}{1-Q_{j}} \right) \left( \prod_{\substack{1 \le j \le n \\ \overline{j} \in I}} \frac{1}{1-Q_{j-1}} \right) \left( {\prod_{j \in I}}^{\star} [\CO_{G/B}(-\ve_{I})] \right). 
\end{align}
Let us take and fix $I \subset [1, \overline{1}]$. For $1 \le j \le n$, we set 
\begin{equation}
\eta_{I}^{\poly}(j) := \begin{cases} 
\dfrac{1}{1-Q_{j}} & \text{if $j \in I$,} \\ 
1 & \text{otherwise.}
\end{cases}
\end{equation}
Then it follows that 
\begin{align}
\zeta_{I}^{\poly}(j) \eta_{I}^{\poly}(j) &= \begin{cases}
1 \cdot \dfrac{1}{1-Q_{j}} & \text{if $j \in I$ and $j+1 \in I$,} \\ 
(1-Q_{j}) \cdot \dfrac{1}{1-Q_{j}} & \text{if $j \in I$ and $j+1 \notin I$,} \\ 
1 \cdot 1 & \text{if $j \notin I$}
\end{cases} \\ 
&= \begin{cases}
\dfrac{1}{1-Q_{j}} & \text{if $j \in I$ and $j+1 \in I$,} \\ 
1 & \text{othewise}
\end{cases} \\ 
&= \vp_{I}^{Q}(j). 
\end{align}
Also, for $2 \le j \le n$, we set 
\begin{equation}
\eta_{I}^{\poly}(\overline{j}) := \begin{cases}
\dfrac{1}{1-Q_{j-1}} & \text{if $\overline{j} \in I$,} \\ 
1 & \text{otherwise.}
\end{cases}
\end{equation}
Then it follows that 
\begin{align}
\zeta_{I}^{\poly}(\overline{j}) \eta_{I}^{\poly}(\overline{j}) &= \begin{cases}
\left( 1 + \dfrac{Q_{j-1} Q_{j} \cdots Q_{n}}{1 - Q_{j-1}} \right) \cdot 1 & \text{if $I = \{ \cdots < j-1 < \overline{j-1} < \cdots \}$,} \\ 
1 \cdot 1 & \text{if $\overline{j} \notin I$ and $I$ is not of the above form,} \\ 
1 \cdot \dfrac{1}{1-Q_{j-1}} & \text{if $\overline{j} \in I$ and $\overline{j-1} \in I$,} \\ 
(1-Q_{j-1}) \cdot \dfrac{1}{1-Q_{j-1}} & \text{if $\overline{j} \in I$ and $\overline{j-1} \notin I$}
\end{cases} \\ 
&= \begin{cases}
1 + \dfrac{Q_{j-1} Q_{j} \cdots Q_{n}}{1 - Q_{j-1}} & \text{if $I = \{ \cdots < j-1 < \overline{j-1} < \cdots \}$,} \\ 
\dfrac{1}{1-Q_{j-1}} & \text{if $\overline{j} \in I$ and $\overline{j-1} \in I$,} \\ 
1 & \text{otherwise}
\end{cases} \\ 
&= \vp_{I}^{Q}(\overline{j}). 
\end{align}
From these, we deduce that 
\begin{align}
\widehat{\Psi}^{Q}(F_{l}) &= \sum_{\substack{I \subset [1, \overline{1}] \\ |I| = l}} \left( \prod_{1 \le j \le \overline{1}} \zeta_{I}^{\poly}(j) \eta_{I}^{\poly}(j) \right) \left( {\prod_{j \in I}}^{\star} [\CO_{G/B}(-\ve_{I})] \right) \\ 
&= \sum_{\substack{I \subset [1, \overline{1}] \\ |I| = l}} \left( \prod_{1 \le j \le \overline{1}} \vp_{I}^{Q}(j) \right) \left( {\prod_{j \in I}}^{\star} [\CO_{G/B}(-\ve_{I})] \right) \\ 
&= \CF_{l},
\end{align}
as desired. 
This proves the lemma. 
\end{proof}

By Corollary~\ref{cor:relation_QK}, we obtain the following. 
\begin{cor}
For $0 \le j \le 2n$, we have $\widehat{\Psi}^{Q}(F_{l}-E_{l}) = 0$. 
\end{cor}

Hence the map $\widehat{\Psi}^{Q}$ induces the $R(T)\bra{Q}$-algebra homomorphism (not yet proved to be an isomorphism) $\Psi^{Q}$, given by \eqref{eq:isomorphism}. 

As a consequence of the proof of Lemma~\ref{lem:QK->Borel}, we obtain an explicit expression, in terms of line bundles, of the Schubert classes $[\CO^{s_{1} \cdots s_{k}}]$ and $[\CO^{s_{1} \cdots s_{n} s_{n-1} \cdots s_{k}}]$ for $1 \le k \le n$. Recall that $z_{\overline{j}} = z_{j}^{-1}$ for $1 \le j \le n$ in the following definition. 

\begin{defn}
\begin{enu}
\item For $1 \le k \le n$ and $0 \le l \le k$, we define $F_{l}^{k} \in (R(T)\bra{Q})[z_{1}^{\pm 1}, \ldots, z_{n}^{\pm 1}]$ by 
\begin{equation}
F_{l}^{k} := \sum_{\substack{I \subset [1, k] \\ |I| = l}} \left( \prod_{1 \le j \le \overline{1}} \zeta_{I}^{\poly}(j) \right) \left( \prod_{j \in I} z_{j} \right). 
\end{equation}
\item For $1 \le k \le n$ and $0 \le l \le 2n-k$, we define $F_{l}^{\overline{k}} \in (R(T)\bra{Q})[z_{1}^{\pm 1}, \ldots, z_{n}^{\pm 1}]$ by 
\begin{equation}
F_{l}^{\overline{k}} := \sum_{\substack{I \subset [1, \overline{k+1}] \\ |I| = l}} \left( \prod_{1 \le j \le \overline{1}} \zeta_{I}^{\poly}(j) \right) \left( \prod_{j \in I} z_{j} \right), 
\end{equation}
where we understand that $\overline{n+1} := n$. 
\end{enu}
\end{defn}

We can prove the following theorem; its proof is the same as that of Lemma~\ref{lem:QK->Borel}. 
\begin{thm}
\begin{enu}
\item For $1 \le k \le n$ and $0 \le l \le k$, we have $\widehat{\Psi}^{Q}(F_{l}^{k}) = \CF_{l}^{k}$. 
\item For $1 \le k \le n$ and $0 \le l \le 2n-k$, we have $\widehat{\Psi}^{Q}(F_{l}^{\overline{k}}) = \CF_{l}^{\overline{k}}$. 
\end{enu}
\end{thm}

Hence, by Corollary~\ref{cor:explicit_classes}, we obtain the following. 
\begin{cor}
\begin{enu}
\item For $1 \le k \le n$, we have 
\begin{equation}
[\CO^{s_{1} \cdots s_{k}}] = \widehat{\Psi}^{Q} \left( \sum_{l = 0}^{k} (-1)^{l} \e^{-l\ve_{1}} F_{l}^{k} \right). 
\end{equation}
\item For $1 \le k \le n$, we have 
\begin{equation}
[\CO^{s_{1} \cdots s_{n} s_{n-1} \cdots s_{k}}] = \widehat{\Psi}^{Q} \left( \sum_{l = 0}^{2n-k} (-1)^{l} \e^{-l\ve_{1}} F_{l}^{\overline{k}} \right). 
\end{equation}
\end{enu}
\end{cor}

\subsection{Finishing the proof of Theorem~\ref{thm:main}}
Let us complete the proof of Theorem~\ref{thm:main}. 
It remains to prove that the $R(T)\bra{Q}$-algebra homomorphism $\Psi{Q}$, given by \eqref{eq:isomorphism}, is an isomorphism. 
For this, we make use of the following lemma, which follows from Nakayama's lemma. 
\begin{lem}[{\cite[Proposition~A.3]{GMSZ}}] \label{lem:quotient_isomorphism_lemma}
Let $R$ be a Noetherian domain, $I \subset R$ an ideal. Assume that $R$ is $I$-adic complete. Let $M$ and $N$ be finitely generated $R$-modules. 
In addition, assume that the $R$-module $N$ and the $(R/I)$-module $N/IN$ are free of the same finite rank. Then, for a homomorphism $f: M \rightarrow N$ of $R$-modules, if the induced homomorphism $\overline{f}: M/IM \rightarrow N/IN$ of $(R/I)$-modules is an isomorphism, then $f$ is also an isomorphism of $R$-modules. 
\end{lem}

We apply this lemma to the case 
\begin{gather}
R = R(T)\bra{Q}, \quad I = (Q_{1}, \ldots, Q_{n}), \\ 
M = (R(T)\bra{Q})[z_{1}^{\pm 1}, \ldots, z_{n}^{\pm 1}]/\CI^{Q}, \quad N = QK_{T}(G/B), \quad \text{and} \quad f = \Psi^{Q}; 
\end{gather}
note that $R$ is $I$-adic complete. 
Recall that $N$ is a free $R$-module of rank $|W| (< \infty)$ since it admits the Schubert basis $\{[\CO_{G/B}^{w}] \mid w \in W\}$. In addition, we see that $N/IN \simeq K_{T}(G/B)$ is also a free $R/I \simeq R(T)$-module of rank $|W|$ since it also admits the Schubert basis. 
Hence it remains to verify the following: 
\begin{enu}
\item $M$ is a finitely generated $R$-module. 
\item The induced $R(T)$-algebra homomorphism $\overline{\Psi^{Q}}: M/IM \rightarrow N/IN$ is an isomorphism. 
\end{enu}

Let $\CI$ be the ideal of $R(T)[z_{1}^{\pm 1}, \ldots, z_{n}^{\pm 1}]$ generated by: 
\begin{equation}
e_{l}(z_{1}, \ldots, z_{n}, z_{n}^{-1}, \ldots, z_{1}^{-1}) - E_{l} \quad \text{for $1 \le l \le n$}. 
\end{equation}
Then, by Remark~\ref{rem:specialization_Fl}, we see that 
\begin{equation} \label{eq:classical_Borel_1}
M/IM \simeq R(T)[z_{1}^{\pm 1}, \ldots, z_{n}^{\pm 1}]/\CI. 
\end{equation}
It is well-known (see, for example, \cite{PR}) that there exists an isomorphism 
\begin{equation} \label{eq:classical_Borel}
\map[\xrightarrow{\sim}]{\Psi}{R(T)[z_{1}^{\pm 1}, \ldots, z_{n}^{\pm 1}]/\CI}{K_{T}(G/B)}{}{z_{j} + \CI}{[\CO_{G/B}(-\ve_{j})],}{\quad 1 \le j \le n}
\end{equation}
of $R(T)$-algebras; 
this is just the classical \emph{Borel presentation}. 
Now part (2) follows from \eqref{eq:classical_Borel}. 
For part (1), we can apply the following lemma. 

\begin{lem}[{\cite[Proposition~A.5\,(1)]{GMSXZZ}}] \label{lem:finitely_generated}
Let $A$ be a Noetherian ring. Let $R := A\bra{Q_{1}, \ldots, Q_{n}}$, $I := (Q_{1}, \ldots, Q_{n}) \subset R$, and $M$ an $R$-module. 
If $M/IM$ is a finitely generated $(R/I)$-module, then $M$ is a finitely generated $R$-module. 
\end{lem}

\begin{prop}
$(R(T)\bra{Q})[z_{1}^{\pm 1}, \ldots, z_{n}^{\pm 1}]/\CI^{Q}$ is a finitely generated $R(T)\bra{Q}$-module. 
\end{prop}
\begin{proof}
We set $M := (R(T)\bra{Q})[z_{1}^{\pm 1}, \ldots, z_{n}^{\pm 1}]/\CI^{Q}$. 
By \eqref{eq:classical_Borel_1} and \eqref{eq:classical_Borel}, $M/IM$ is a free $R(T)$-module of rank $|W| (< \infty)$; in particular, $M/IM$ is a finitely generated $R(T)$-module. 
Therefore, by applying Lemma~\ref{lem:finitely_generated} for the case $A = R(T)$, we conclude that $M = (R(T)\bra{Q})[z_{1}^{\pm 1}, \ldots, z_{n}^{\pm 1}]/\CI^{Q}$ is a finitely generated $R = R(T)\bra{Q}$-module. 
This proves the proposition. 
\end{proof}

Thus, by Lemma~\ref{lem:quotient_isomorphism_lemma}, we can complete the proof of Theorem~\ref{thm:main}. 

\appendix

\section*{Appendices}

\section{Proofs of Theorems~\ref{thm:IC_1} and \ref{thm:IC_2}} \label{sec:IC_proof}

In this appendix, we prove Theorems~\ref{thm:IC_1} and \ref{thm:IC_2}. 

First, let us recall the definition of the quantum Bruhat graph. 

\begin{defn}[{\cite[Definition~6.1]{BFP}}]
Let $W$ be the Weyl group of $G$ (of an arbitrary type). The \emph{quantum Bruhat graph} $\QBG(W)$ on $W$ is the $\Delta^{+}$-labeled directed graph with vertex set $W$ and directed edges $x \xrightarrow{\alpha} x s_{\alpha}$ for $x \in W$ and $\alpha \in \Delta^{+}$ such that either of the following (B) or (Q) holds: 
\begin{itemize}
\item[(B)] $\ell(y) = \ell(x) + 1$, or 
\item[(Q)] $\ell(y) = \ell(x) - 2\pair{\rho}{\alpha^{\vee}} + 1$, 
\end{itemize}
where $\rho := (1/2) \sum_{\alpha \in \Delta^{+}} \alpha$. 
If (B) (resp., (Q)) holds, then the edge $x \xrightarrow{\alpha} x s_{\alpha}$ is called a \emph{Bruhat} (resp., \emph{quantum}) \emph{edge}. 
\end{defn}

For $G$ of type $C$, i.e., for $G = \Sp$, we know the following useful criterion for the directed edges in $\QBG(W)$; for $1 \le j \le n$, we set $\sign(j) := 1$ and $\sign(\overline{j}) := -1$. 

\begin{prop}[cf. {\cite[Proposition~5.7]{Len}}] \label{prop:criterion}
In the above setting, let $w \in W$. 
\begin{enu}
\item For $1 \le i < j \le n$, we have the Bruhat edge $w \xrightarrow{(i, j)} ws_{(i, j)}$ in $\QBG(W)$ if and only if $w(i) < w(j)$ and there does not exist any $i < k < j$ such that $w(i) < w(k) < w(j)$. 

\item For $1 \le i < j \le n$, we have the quantum edge $w \xrightarrow{(i, j)} ws_{(i, j)}$ in $\QBG(W)$ if and only if $w(i) > w(j)$ and all $i < k < j$ satisfy $w(i) > w(k) > w(j)$. 

\item For $1 \le i < j \le n$, we have the edge $w \xrightarrow{(i, \overline{j})} ws_{(i, \overline{j})}$ in $\QBG(W)$ if and only if $w(i) < w(\overline{j})$, $\sign(w(i)) = \sign(w(\overline{j}))$, and there does not exist any $i < k < \overline{j}$ such that $w(i) < w(k) < w(\overline{j})$. 
In this case, the edge $w \xrightarrow{(i, \overline{j})} ws_{(i, \overline{j})}$ is a Bruhat edge. 

\item For $1 \le i \le n$, we have the Bruhat edge $w \xrightarrow{(i, \overline{i})} ws_{(i, \overline{i})}$ in $\QBG(W)$ if and only if $w(i) < w(\overline{i})$ and there does not exist any $i < k < \overline{i}$ such that $w(i) < w(k) < w(\overline{i})$. 

\item For $1 \le l \le n$, we have the quantum edge $w \xrightarrow{(i, \overline{i})} ws_{(i, \overline{i})}$ in $\QBG(W)$ if and only if $w(i) > w(\overline{i})$ and all $i < k < \overline{i}$ satisfy $w(i) > w(k) > w(\overline{i})$. 
\end{enu}
\end{prop}

Let $G$ be of an arbitrary type. 
Let us briefly review the (``generalized'') quantum alcove model, introduced by \cite{LL} (see also \cite{LNS}). 

We set

\begin{equation}
|\alpha| := \begin{cases}
\alpha & \text{if $\alpha \in \Delta^{+}$,} \\ 
-\alpha & \text{if $\alpha \in -\Delta^{+}$.} 
\end{cases} 
\end{equation}

\begin{defn}[{\cite[Definition~17]{LNS}}]
Let $\Gamma = (\gamma_{1}, \ldots, \gamma_{r})$ be a sequence of roots, and $w \in W$. 
A subset $A = \{i_{1}, \ldots, i_{s}\} \subset \{1, \ldots, r\}$ is said to be \emph{$w$-admissible} if the sequence 
\begin{equation}
\Pi(w, A): w = w_{0} \xrightarrow{|\gamma_{i_{1}}|} w_{1} \xrightarrow{|\gamma_{i_{2}}|} \cdots \xrightarrow{|\gamma_{i_{s}}|} w_{s}
\end{equation}
is a directed path in $\QBG(W)$; 
we set 
\begin{align}
\ed(A) &:= w_{s}, \\ 
\down(A) &:= \sum_{\substack{1 \le j \le s \\ \text{$w_{j-1} \rightarrow w_{j}$ is a quantum edge}}} |\gamma_{i_{j}}|^{\vee}. 
\end{align}
Let $\CA(w, \Gamma)$ denote the set of $w$-admissible subsets associated to $\Gamma$. 
\end{defn}

We will regard a $w$-admissible subset $A \in \CA(w, \Gamma)$ as a subset of $\Gamma$, and write it as $A = \{\gamma_{i_{1}}, \ldots, \gamma_{i_{s}}\}$. 
Also, for a tuple $(A_{1}, \ldots, A_{r})$ of admissible subsets $A_{j} \in \CA(w_{j}, \Gamma_{j})$, $1 \le j \le r$, with $w_{j} \in W$ and $\Gamma_{j}$ a sequence of roots, we set 
\begin{equation}
\down(A_{1}, \ldots, A_{r}) := \down(A_{1}) + \cdots + \down(A_{r}). 
\end{equation}

In the following, we assume that $G = \Sp$. 
We use admissible subsets associated to the following two types of sequences of roots for $1 \le k \le n$: 
\begin{align}
\Theta_{k} &:= (-(1, k), \ldots, -(k-1, k)); \\ 
\begin{split}
\Gamma_{k}(k) &:= (-(1, \overline{k}), \ldots, -(k-1, \overline{k}), \\ 
& \qquad {-(k, \overline{k+1})}, \ldots, -(k, \overline{n}), \\ 
& \qquad {-(k, \overline{k})}, \\ 
& \qquad {-(k, n)}, \ldots -(k, k+1)). 
\end{split}
\end{align}

Let us briefly recall from \cite[\S 4]{KNO} (the ``second half'' of) the inverse Chevalley formula in type $C$.
Following \cite[\S 4.1]{KNO}, for $j, m \in [1, \overline{1}]$, with $j < m$, we set 
\begin{equation}
\CS_{m, j} := \{ (j_{1}, \ldots, j_{r}) \mid r \ge 1, \ j_{1}, \ldots, j_{r} \in [1, \overline{1}], \ m > j_{1} > \cdots > j_{r} = j \}. 
\end{equation}
Also, for $w \in W$ and $1 \le k \le l \le n$, we set 
\begin{equation}
\CA_{w}^{k, l} := \{ A \in \CA(w, \Theta_{k}) \setminus \{ \emptyset \} \mid \ed(A)^{-1}w\ve_{k} = \ve_{l} \}, 
\end{equation}
while for $w \in W$, and $1 \le k \le n$, $1 \le l \le \overline{k}$, we set 
\begin{equation}
\CA_{w}^{\overline{k}, l} := \{ A \in \CA(w, \Gamma_{k}(k)) \setminus \{ \emptyset \} \mid \ed(A)^{-1}w(-\ve_{k}) = \ve_{l} \}. 
\end{equation}

\begin{thm}[{\cite[Theorem~4.3]{KNO} combined with \cite[Theorem~5.8]{KNS}}] \label{thm:IC_2nd-half}
For $w \in W$ and $m = 1, \ldots, n$, the following identity holds in $\KTCQG$: 
\begin{equation} \label{eq:IC_2nd-half}
\begin{split}
& \e^{-w_{\ve_{m}}} [\CO_{\bQG(w)}] \\ 
&= \sum_{B \in \CA(w, \Theta_{m})} (-1)^{|B|} \left[ \CO_{\bQG \left( \ed(B)t_{\down(B)} \right)}(-\ve_{m}) \right] \\ 
& \quad + \sum_{j = m+1}^{n} \sum_{(j_{1}, \ldots, j_{r}) \in \CS_{\overline{m}, \overline{j}}} \sum_{A_{1} \in \CA_{w}^{\overline{m}, j_{1}}} \cdots \sum_{A_{r} \in \CA_{\ed(A_{r-1})}^{j_{r-1}, j_{r}}} (-1)^{|A_{1}| + \cdots + |A_{r}| - r} q^{-\pair{\ve_{j}}{\down(A_{1}, \ldots, A_{r})}} \\ 
& \hspace{35mm} \times \sum_{B \in \CA(\ed(A_{r}), \Theta_{j})} (-1)^{|B|} \left[ \CO_{\bQG \left( \ed(B)t_{\down(A_{1}, \ldots, A_{r}, B)} \right)}(-\ve_{j}) \right] \\ 
& \quad + \sum_{j = 1}^{n} \sum_{(j_{1}, \ldots, j_{r}) \in \CS_{\overline{m}, j}} \sum_{A_{1} \in \CA_{w}^{\overline{m}, j_{1}}} \cdots \sum_{A_{r} \in \CA_{\ed(A_{r-1})}^{j_{r-1}, j_{r}}} (-1)^{|A_{1}| + \cdots + |A_{r}| - r} q^{\pair{\ve_{j}}{\down(A_{1}, \ldots, A_{r})}} \\ 
& \hspace{35mm} \times \sum_{B \in \CA(\ed(A_{r}), \Gamma_{j}(j))} (-1)^{|B|} \left[ \CO_{\bQG \left( \ed(B)t_{\down(A_{1}, \ldots, A_{r}, B)} \right)}(\ve_{j}) \right]. 
\end{split}
\end{equation}
\end{thm}

We will prove Theorem~\ref{thm:IC_2} by using Theorem~\ref{thm:IC_2nd-half}; 
since the proof of Theorem~\ref{thm:IC_1} is similar (use \cite[Theorem~4.5]{KNO} instead of \cite[Theorem~4.3]{KNO}) and easier, we leave it to the reader (cf. the proof of \cite[Proposition~4.5]{MNS}). 

\begin{proof}[Proof of Theorem~\ref{thm:IC_2}]
We put $w = s_{1} \cdots s_{n-1} s_{n} s_{n-1} \cdots s_{k}$ and $m = k$ in Theorem~\ref{thm:IC_2nd-half}. 
Note that $w = [2, 3, \ldots, k, \overline{1}, k+1, \ldots, n]$ in ``window notation''. 
Then, by Proposition~\ref{prop:criterion}, we see that 
\begin{equation} \label{eq:admissible_mountain2}
\CA(s_{1} \cdots s_{n-1} s_{n} s_{n-1} \cdots s_{k}, \Theta_{k}) = \{ \emptyset, \{-(k-1, k)\} \}. 
\end{equation}
Hence the first sum on the RHS of \eqref{eq:IC_2nd-half} is 
\begin{align}
& \sum_{B \in \CA(s_{1} \cdots s_{n-1} s_{n} s_{n-1} \cdots s_{k}, \Theta_{k})} (-1)^{|B|} \left[ \CO_{\bQG \left( \ed(B)t_{\down(B)} \right)} (-\ve_{k}) \right] \\ 
&= \underbrace{[\CO_{\bQG(s_{1} \cdots s_{n-1} s_{n} s_{n-1} \cdots s_{k})} (-\ve_{k})]}_{B = \emptyset} - \underbrace{[\CO_{\bQG(s_{1} \cdots s_{n-1} s_{n} s_{n-1} \cdots s_{k-1})} (-\ve_{k})]}_{B = \{-(k-1, k)\}}. \label{eq:IC_2nd-half_1st-sum}
\end{align}
In addition, we see that 
\begin{equation} \label{eq:admissible_mountain3}
\CA(s_{1} \cdots s_{n-1} s_{n} s_{n-1} \cdots s_{k}, \Gamma_{k}(k)) = \{ \emptyset, \{-(k, k+1)\}, \{-(k, \overline{k})\}, \{-(k, \overline{k}), -(k, k+1)\} \}, \\ 
\end{equation}
with 
\begin{align}
\ed(\{-(k, k+1)\}) &= s_{1} \cdots s_{n-1} s_{n} s_{n-1} \cdots s_{k+1}, \\ 
\ed(\{-(k, \overline{k})\}) &= s_{1} \cdots s_{k-1}, \\ 
\ed(\{-(k, \overline{k}), -(k, k+1)\}) &= s_{1} \cdots s_{k}, 
\end{align}
and 
\begin{align}
\down(\{-(k, k+1)\}) &= \alpha_{k}^{\vee}, \\ 
\down(\{-(k, \overline{k})\}) &= \alpha_{k}^{\vee} + \cdots + \alpha_{n}^{\vee}, \\ 
\down(\{-(k, \overline{k}), -(k, k+1)\}) &= \alpha_{k}^{\vee} + \cdots + \alpha_{n}^{\vee}. 
\end{align}
From these, we deduce that the remaining terms in the second sum on the RHS of \eqref{eq:IC_2nd-half} correspond to the tuples $(j, (j_{1}, \ldots, j_{r}), A_{1}, \ldots, A_{r})$ such that 
\begin{itemize}
\item $k+1 \le j \le n$, 
\item $(j_{1}, \ldots, j_{r}) = (\overline{k}, \overline{k+1}, \ldots, \overline{j})$, and 
\item $(A_{1}, \ldots, A_{r}) = (\{-(k, k+1)\}, \{-(k+1, k+2)\}, \ldots, \{-(j-1, j)\})$, with 
\begin{align}
\ed(A_{r}) &= s_{1} \cdots s_{n-1} s_{n} s_{n-1} \cdots s_{j}, \\ 
\down(A_{1}, \ldots, A_{r}) &= \alpha_{k}^{\vee} + \cdots + \alpha_{j-1}^{\vee} \\ 
|A_{1}| + \cdots + |A_{r}| &= r = j-k. 
\end{align}
\end{itemize}
Therefore, the second sum on the RHS of \eqref{eq:IC_2nd-half} is 
\begin{align}
& \sum_{j = k+1}^{n} q^{-\pair{\ve_{j}}{\alpha_{k}^{\vee} + \cdots + \alpha_{j-1}^{\vee}}} \sum_{B \in \CA(s_{1} \cdots s_{n-1} s_{n} s_{n-1} \cdots s_{j}, \Theta_{j})} (-1)^{|B|} \left[ \CO_{\bQG \left( \ed(B)t_{\alpha_{k}^{\vee} + \cdots + \alpha_{j-1}^{\vee}} \right)}(-\ve_{j}) \right] \\ 
\begin{split}
&= q \sum_{j = k+1}^{n} \left( \underbrace{\left[\CO_{\bQG \left( s_{1} \cdots s_{n-1}s_{n}s_{n-1} \cdots s_{j}t_{\alpha_{k}^{\vee} + \cdots + \alpha_{j-1}^{\vee}} \right)} (-\ve_{j}) \right]}_{B = \emptyset} \right. \\ 
& \hspace{20mm} \left. - \underbrace{\left[ \CO_{\bQG \left( s_{1} \cdots s_{n-1}s_{n}s_{n-1} \cdots s_{j-1}t_{\alpha_{k}^{\vee} + \cdots + \alpha_{j-1}^{\vee}} \right)} (-\ve_{j}) \right]}_{B = \{-(j-1, j)\}} \right), 
\end{split} \label{eq:IC_2nd-half_2nd-sum}
\end{align}
where we have used \eqref{eq:admissible_mountain2} for the above equality. 

Also, for $1 \le i \le n$, we see that 
\begin{equation}
\CA(s_{1} \cdots s_{i}, \Theta_{i}) = \{ \emptyset, \{-(i, i+1)\}\}, 
\end{equation}
with 
\begin{equation}
\ed(\{-(i, i+1)\}) = s_{1} \cdots s_{i-1}, \quad \down(\{-(i, i+1)\}) = \alpha_{i}^{\vee}. 
\end{equation}
By combining this with \eqref{eq:admissible_mountain3}, we deduce that 
the remaining terms in the third sum on the RHS of \eqref{eq:IC_2nd-half} correspond to the tuples $(j, (j_{1}, \ldots, j_{r}), A_{1}, \ldots, A_{r})$ such that 
\begin{itemize}
\item $j = 1, \ldots, n$, 
\item $(j_{1}, \ldots, j_{r}) = (\overline{k}, \overline{k+1}, \ldots, \overline{l}, l, l-1, \ldots, j)$ for some $k \le l \le n$ with $j \le l$, \\ 
or $(j_{1}, \ldots, j_{r}) = (\overline{k}, \overline{k+1}, \ldots, \overline{l-1}, l, l-1, \ldots, j)$ for some $k < l \le n$ with $j \le l$, 
\item if $(j_{1}, \ldots, j_{r}) = (\overline{k}, \overline{k+1}, \ldots, \overline{l}, l, l-1, \ldots, j)$, then 
\begin{equation}
\begin{split}
(A_{1}, \ldots, A_{r}) &= (\{-(k, k+1)\}, \ldots, \{-(l-1, l)\}, \{-(l, \overline{l})\}, \\ 
& \qquad \{-(l-1, l)\}, \ldots, \{-(j, j+1)\}), 
\end{split}
\end{equation}
with 
\begin{align}
\ed(A_{r}) &= s_{1} \cdots s_{j-1}, \\ 
\down(A_{1}, \ldots, A_{r}) &= (\alpha_{k}^{\vee} + \cdots \alpha_{l-1}^{\vee}) + (\alpha_{l}^{\vee} + \cdots + \alpha_{n}^{\vee}) + (\alpha_{l-1}^{\vee} + \cdots + \alpha_{j}^{\vee}), \\ 
|A_{1}| + \cdots + |A_{r}| &= r = (l-k) + 1 + (l-j) = 2l-k-j+1, 
\end{align}

\item if $(j_{1}, \ldots, j_{r}) = (\overline{k}, \overline{k+1}, \ldots, \overline{l-1}, l, l-1, \ldots, j)$, then 
\begin{equation}
\begin{split}
(A_{1}, \ldots, A_{r}) &= (\{-(k, k+1)\}, \ldots, \{-(l-2, l-1)\}, \{-(l-1, \overline{l-1}), -(l-1, l)\}, \\ 
& \qquad \{-(l-1, l)\}, \ldots, \{-(j, j+1)\}), 
\end{split}
\end{equation}
with 
\begin{align}
\ed(A_{r}) &= s_{1} \cdots s_{j-1}, \\ 
\down(A_{1}, \ldots, A_{r}) &= (\alpha_{k}^{\vee} + \cdots \alpha_{l-2}^{\vee}) + (\alpha_{l-1}^{\vee} + \cdots + \alpha_{n}^{\vee}) + (\alpha_{l-1}^{\vee} + \cdots + \alpha_{j}^{\vee}), \\ 
|A_{1}| + \cdots + |A_{r}| &= r+1 = (l-k-1) + 2 + (l-j+1) = 2l-k-j+2. 
\end{align}
\end{itemize}
Here observe that for fixed $1 \le j \le n$ and $k < l \le n$ with $j \le l$, the term in the third sum on the RHS of \eqref{eq:IC_2nd-half} corresponding to $(j_{1}, \ldots, j_{r}) = (\overline{k}, \overline{k+1}, \ldots, \overline{l}, l, l-1, \ldots, j)$ and that corresponding to $(j_{1}, \ldots, j_{r}) = (\overline{k}, \overline{k+1}, \ldots, \overline{l-1}, l, l-1, \ldots, j)$ cancel each other out. 
Hence only the terms corresponding to $j \le k$, $(j_{1}, \ldots, j_{r}) = (\overline{k}, k, k-1, \ldots, j)$ remain uncancelled. 
Therefore, the third sum on the RHS of \eqref{eq:IC_2nd-half} is 
\begin{align}
& \sum_{j = 1}^{k} q^{\pair{\ve_{j}}{\alpha_{j}^{\vee} + \cdots + \alpha_{n}^{\vee}}} \sum_{B \in \CA(s_{1} \cdots s_{j-1}, \Gamma_{j}(j))} (-1)^{|B|} \left[ \CO_{\bQG \left( \ed(B)t_{\down(B) + \alpha_{j}^{\vee} + \cdots + \alpha_{n}^{\vee}} \right)} (\ve_{j}) \right] \\ 
&= q \sum_{j = 1}^{k} \left( \underbrace{\left[ \CO_{\bQG \left( s_{1} \cdots s_{j-1}t_{\alpha_{j}^{\vee} + \cdots + \alpha_{n}^{\vee}} \right)} (\ve_{j}) \right]}_{B = \emptyset} - \underbrace{\left[ \CO_{\bQG \left( s_{1} \cdots s_{j}t_{\alpha_{j}^{\vee} + \cdots + \alpha_{n}^{\vee}} \right)} (\ve_{j}) \right]}_{B = \{-(j, j+1)\}} \right); \label{eq:IC_2nd-half_3rd-sum}
\end{align}
for this equality, we have used the fact that
\begin{equation}
\CA(s_{1}, \ldots, s_{j-1}, \Gamma_{j}(j)) = \{ \emptyset, \{-(j, j+1)\} \}, 
\end{equation}
which can be shown by using Proposition~\ref{prop:criterion}.  

By combining \eqref{eq:IC_2nd-half_1st-sum}, \eqref{eq:IC_2nd-half_2nd-sum}, and \eqref{eq:IC_2nd-half_3rd-sum}, we now complete the proof of the theorem. 
\end{proof}

\section{Proof of Proposition~\ref{prop:symmetry}} \label{sec:proof_symmetry}
In this appendix, we give a proof of Proposition~\ref{prop:symmetry}. 

\begin{defn}
For $A, B \subset [1, n]$ such that $A \cap B = \emptyset$, and for $0 \le k \le 2n$, we set 
\begin{equation}
\CJ_{A, B}^{k} := \{ I \subset [1, \overline{1}] \mid \text{$[\CO_{\bQG}(-\ve_{I})] = [\CO_{\bQG}(-\ve_{A} + \ve_{B})]$ and $|I| = k$} \}. 
\end{equation}
\end{defn}

We have 
\begin{equation}
\FF_{k} = \sum_{\substack{A, B \subset [1, n] \\ A \cap B = \emptyset \\ k - (|A|+|B|) \in 2\BZ_{\ge 0}}} \left( \sum_{I \in \CJ_{A, B}^{l}} \left( \prod_{1 \le j \le \overline{1}} \psi_{I}(j) \right) \right) [\CO_{\bQG}(-\ve_{A} + \ve_{B})]. 
\end{equation}

If $k < n$, then we can construct a bijection $\CJ_{A, B}^{k} \rightarrow \CJ_{A, B}^{2n-k}$ as follows. 
Let $A, B \subset [1, n]$ be such that $A \cap B = \emptyset$. We set $s := |A|$ , $t := |B|$, and write $A$, $B$ as: 
\begin{equation}
A = \{i_{1} < \cdots < i_{s}\}, \quad B = \{j_{1} < \cdots < j_{t}\}. 
\end{equation}
We consider only the case $\CJ_{A, B}^{k} \not= \emptyset$. 
Let $k < n$ be such that $s+t \le k$ and $k - (s+t) \in 2\BZ$; 
we set $r := (k-(s+t))/2$. If $I \in \CJ_{A, B}^{k}$, then there exist $1 \le k_{1} < \cdots < k_{r} \le n$ such that 
\begin{equation} \label{eq:JABk_decomposition}
I = \{i_{1}, \ldots, i_{s}\} \sqcup \{\overline{j_{1}}, \ldots, \overline{j_{t}}\} \sqcup \{k_{1}, \ldots, k_{r}\} \sqcup \{\overline{k_{1}}, \ldots, \overline{k_{r}}\}. 
\end{equation}
Let us take $m_{1}, \ldots, m_{u} \in [1, n]$ such that 
\begin{equation}
\{ m_{1} < \cdots < m_{u} \} = [1, n] \setminus (A \sqcup B \sqcup \{k_{1}, \ldots, k_{r}\}), 
\end{equation}
and define $I^{\ast} \subset [1, \overline{1}]$ by 
\begin{equation}
I^{\ast} := \{i_{1}, \ldots, i_{s}\} \sqcup \{\overline{j_{1}}, \ldots, \overline{j_{t}}\} \sqcup \{m_{1}, \ldots, m_{u}\} \sqcup \{\overline{m_{1}}, \ldots, \overline{m_{u}}\}. 
\end{equation}
It is easy to check that $I^{\ast} \in \CJ_{A, B}^{2n-k}$. 

\begin{exm}
Let $n = 7$ and $I = \{2, 4, 5, \overline{6}, \overline{5}, \overline{2}\}$. We see that $I \in \CJ_{\{4\}, \{6\}}^{6}$. By decomposing $I$ as: 
\begin{equation}
I = \{4\} \sqcup \{\overline{6}\} \sqcup \{2, 5\} \sqcup \{\overline{2}, \overline{5}\}, 
\end{equation}
we obtain 
\begin{equation}
I^{\ast} = \{4\} \sqcup \{\overline{6}\} \sqcup \{1, 3, 7\} \sqcup \{\overline{1}, \overline{3}, \overline{7}\} \in \CJ_{\{4\}, \{6\}}^{8}. 
\end{equation}
\end{exm}

We can easily verify the following lemma. 
\begin{lem}
The assignment $I \mapsto I^{\ast}$ for $I \in \CJ_{A, B}^{k}$ gives a bijection $\CJ_{A, B}^{k} \xrightarrow{\sim} \CJ_{A, B}^{2n-k}$. 
\end{lem}

In order to prove Proposition~\ref{prop:symmetry}, we need the following lemma. 
\begin{lem} \label{lem:duality}
Let $I \subset [1, \overline{1}]$ be decomposed as in \eqref{eq:JABk_decomposition}. Assume that $k_{r} < i_{s}$, $k_{r} < j_{t}$, or $\{\max\{i_{s}, j_{t}\}+1, \ldots, n\} \subset I$. 
Then we have 
\begin{equation}
\prod_{1 \le j \le \overline{1}} \psi_{I}(j) = \prod_{\substack{1 \le j \le n \\ j \notin I, \ j+1 \in I}} (1 - \st_{j}) \prod_{\substack{2 \le j \le n \\ \overline{j} \notin I, \ \overline{j-1} \in I}} (1 - \st_{j-1}) = \prod_{1 \le j \le \overline{1}} \psi_{I^{\ast}}(j), 
\end{equation}
where we understand that $n+1 := \overline{n}$. 
\end{lem}
\begin{proof}
The first equality follows from the definition of $\psi_{I}$. 
We show the second equality. For $1 \le j \le n-1$, we see that $j \notin I^{\ast}$ is equivalent to $\overline{j} \in I$, while $j+1 \in I^{\ast}$ is equivalent to $\overline{j+1} \notin I$. Also, the condition that $n \notin I^{\ast}$ and $\overline{n} \in I^{\ast}$ is equivalent to the condition that $n \notin I$ and $\overline{n} \in I$. In addition, for $2 \le j \le n$, we see that $\overline{j} \notin I^{\ast}$ is equivalent to $j \in I$, while $\overline{j-1} \in I^{\ast}$ is equivalent to $j-1 \notin I$. This shows the second equality. This proves the lemma.
\end{proof}

\begin{proof}[The sketch of the proof of Proposition~\ref{prop:symmetry}]
Take $A, B \subset [1, n]$ such that $A \cap B = \emptyset$. We set $s := |A|$, $t := |B|$, and assume that $s+t \le k$, $k - (s+t) \in 2\BZ$ (so that $\CJ_{A, B}^{k} \not= \emptyset$). 
We write $A$ and $B$ as: 
\begin{equation}
A = \{i_{1} < \cdots < i_{s}\}, \quad B = \{j_{1} < \cdots < j_{t}\}, 
\end{equation}
and set $M := \max\{i_{s}, j_{t}\}$ (if $A = B = \emptyset$, then we set $M := 0$). 
It suffices to prove that 
\begin{equation} \label{eq:duality}
\sum_{I \in \CJ_{A, B}^{k}} \left( \prod_{1 \le j \le \overline{1}} \psi_{I}(j) \right) = \sum_{J \in \CJ_{A, B}^{2n-k}} \left( \prod_{1 \le j \le \overline{1}} \psi_{J}(j) \right). 
\end{equation}

For $p \ge 0$, we set $\CJ_{A, B}^{k}(p)$
\begin{equation}
\CJ_{A, B}^{k}(p) := \{ I \in \CJ_{A, B}^{k} \mid |I \cap \{M+1, \ldots, n\}| = p \}. 
\end{equation}
For $0 \le p \le n - M$, if we can show that 
\begin{equation} \label{eq:duality_2}
\sum_{I \in \CJ_{A, B}^{k}(p)} \left( \prod_{1 \le j \le \overline{1}} \psi_{I}(j) \right) = \sum_{I \in \CJ_{A, B}^{k}(p)} \left( \prod_{1 \le j \le \overline{1}} \psi_{I^{\ast}}(j) \right), 
\end{equation}
then \eqref{eq:duality} follows. 
If $I \in \CJ_{A, B}^{k}$ is decomposed as in \eqref{eq:JABk_decomposition} in such a way that 
$k_{r} < i_{s}$, $k_{r} < j_{t}$, or $\{M+1, \ldots, n\} \subset I$, 
then it follows from Lemma~\ref{lem:duality} that 
\begin{equation}
\prod_{1 \le j \le \overline{1}} \psi_{I}(j) = \prod_{1 \le j \le \overline{1}} \psi_{I^{\ast}}(j). 
\end{equation}
This shows \eqref{eq:duality_2} for $p = 0, n - M$. 
Let us consider the case that $1 \le p \le n - M - 1$. 
We see that 
\begin{align}
\sum_{I \in \CJ_{A, B}^{k}(p)} \left( \prod_{1 \le j \le \overline{1}} \psi_{I}(j) \right) &= \sum_{J \in \CJ_{A, B}^{k-2p}(0)} \sum_{M < k_{1} < \cdots < k_{p} \le n} \left( \prod_{1 \le j \le \overline{1}} \psi_{J \sqcup \{k_{1}, \ldots, k_{p}, \overline{k_{1}}, \ldots, \overline{k_{p}}\}} (j) \right), \\ 
\sum_{J \in \CJ_{A, B}^{k}(p)} \left( \prod_{1 \le j \le \overline{1}} \psi_{I^{\ast}}(j) \right) &= \sum_{I \in \CJ_{A, B}^{k-2p}(0)} \sum_{M < k_{1} < \cdots < k_{p} \le n} \left( \prod_{1 \le j \le \overline{1}} \psi_{(J \sqcup \{k_{1}, \ldots, k_{p}, \overline{k_{1}}, \ldots, \overline{k_{p}}\})^{\ast}} (j) \right). 
\end{align}

Now, let $A, B \subset [1, n]$ be such that $A \cap B = \emptyset$, and set $M := \max (A \sqcup B)$. For $1 \le p \le n-M-1$ and $J \in \CJ_{A, B}^{k-2p}(0)$, we set 
\begin{align}
S(J, p) &:= \sum_{M < k_{1} < \cdots < k_{p} \le n} \left( \prod_{1 \le j \le \overline{1}} \psi_{J \sqcup \{k_{1}, \ldots, k_{p}, \overline{k_{1}}, \ldots, \overline{k_{p}}\}} (j) \right), \\ 
T(J, p) &:= \sum_{M < k_{1} < \cdots < k_{p} \le n} \left( \prod_{1 \le j \le \overline{1}} \psi_{(J \sqcup \{k_{1}, \ldots, k_{p}, \overline{k_{1}}, \ldots, \overline{k_{p}}\})^{\ast}} (j) \right). 
\end{align}
We can show that 
\begin{equation} \label{eq:duality_3}
S(J, p) = T(J, p)
\end{equation}
for all $A, B \subset [1, n]$ such that $A \cap B = \emptyset$, $1 \le p \le n-M-1$, and $J \in \CJ_{A, B}^{k-2p}(0)$ in the following steps (the details are left to the reader): 
\begin{description}
\item[\fbox{Step~1}] Show \eqref{eq:duality_3} for $A, B \subset [1, n]$ such that $A \cap B = \emptyset$, and $J \in \CJ_{A, B}^{k-2p}(0)$ with $M \in J$ by induction on $p \ge 1$. 
\item[\fbox{Step~2}] Show \eqref{eq:duality_3} for $A, B \subset [1, n]$ such that $A \cap B = \emptyset$, and $J \in \CJ_{A, B}^{k-2p}(0)$ with $\overline{M} \in J$ by induction on $p \ge 1$. 
\item[\fbox{Step~3}] Show \eqref{eq:duality_3} for $J = \emptyset$ by induction on $p \ge 1$. 
\end{description}
This completes the sketch of the proof of the proposition. 
\end{proof}


\begin{thebibliography}{GMSXZZ}
\bibitem[ACT]{ACT} D.~Anderson, L.~Chen, and H.-H.~Tseng, On the quantum $K$-ring of the flag manifold, arXiv:1711.08414.
\bibitem[BFP]{BFP} F.~Brenti, S.~Fomin, and A.~Postnikov, Mixed Bruhat operators and Yang--Baxter equations for Weyl groups, \textit{Int. Math. Res. Not.} \textbf{1999} (1999), no.~8, 419--441. 
\bibitem[Giv]{Givental} A.~Givental, On the WDVV equation in quantum $K$-theory, \textit{Michigan Math. J.} \textbf{48} (2000), 295--304. 
\bibitem[GK]{GK} A.~Givental and B.~Kim, Quantum cohomology of flag manifolds and Toda lattices, \textit{Commun. Math. Phys.} \textbf{168} (1995), no.~3, 609--641. 
\bibitem[GMSZ1]{GMSZ} W.~Gu, L.C.~Mihalcea, E.~Sharpe, and H.~Zou, Quantum K theory of Grassmannians, Wilson line operators, and Schur bundles, arXiv:2208.01091. 
\bibitem[GMSZ2]{GMSZ2} W.~Gu, L.~Mihalcea, E.~Sharpe, and H.~Zou, Quantum K theory of symplectic Grassmannians, \textit{J. Geom. Phys.} \textbf{177} (2022), Paper No. 104548, 38~pp. 
\bibitem[GMSXZZ]{GMSXZZ} W.~Gu, L.C.~Mihalcea, E.~Sharpe, W.~Xu, H.~Zhang, and H.~Zou, Quantum K Whitney relations for partial flag varieties, arXiv:2310.03826. 
\bibitem[Kat]{Kato} S.~Kato, Loop structure on equivariant $K$-theory of semi-infinite flag manifolds, arXiv:1805.01718. 
\bibitem[Kim]{Kim} B.~Kim, Quantum cohomology of flag manifolds $G/B$ and quantum Toda lattices, \textit{Ann. Math.} \textbf{149} (1999), no.~1, 129--148. 
\bibitem[KNO]{KNO} T.~Kouno, S.~Naito, and D.~Orr, Identities of inverse Chevalley type for the graded characters of level-zero Demazure submodules over quantum affine algebras of type $C$, \textit{Algebr. Represent. Theory} \textbf{27} (2024), no.~1, 429--460. 
\bibitem[KNOS]{KNOS} T.~Kouno, S.~Naito, D.~Orr, and D.~Sagaki, Inverse $K$-Chevalley formulas for semi-infinite flag manifolds, I: minuscule weights in ADE type, \textit{Forum Math. Sigma} \textbf{9} (2021), Paper No. e51, 25~pp. 
\bibitem[KNS]{KNS} S.~Kato, S.~Naito, and D.~Sagaki, Equivariant $K$-theory of semi-infinite flag manifolds and the Pieri--Chevalley formula, \textit{Duke Math. J.} \textbf{169} (2020), no.~13, 2421--2500. 
\bibitem[Lee]{Lee} Y.-P. Lee, Quantum $K$-theory, I: Foundations, \textit{Duke Math. J.} \textbf{121} (2004), no.~3, 389--424. 
\bibitem[Len]{Len} C.~Lenart, From Macdonald polynomials to a charge statistic beyond type $A$, \textit{J. Combin. Theory Ser. A} \textbf{119} (2012), no,~3, 683--712. 
\bibitem[LL]{LL} C.~Lenart and A.~Lubovsky, A generalization of the alcove model and its applications, \textit{J. Algebr. Comb.} \textbf{41} (2015), no.~3, 751--783. 
\bibitem[LM]{LM} C.~Lenart and T.~Maeno, Quantum Grothendieck Polynomials, arXiv:math/0608232. 
\bibitem[LNOS]{LNOS} C.~Lenart, S.~Naito, D.~Orr, and D.~Sagaki, Inverse $K$-Chevalley formulas for semi-infinite flag manifolds, II: arbitrary weights in ADE type, \textit{Adv. Math.} \textbf{423} (2023), Paper No. 109037, 63~pp. 
\bibitem[LNS]{LNS} C.~Lenart, S.~Naito, and D.~Sagaki, A general Chevalley formula for semi-infinite flag manifolds and quantum $K$-theory, \textit{Selecta Math. (N.S.)} \textbf{30} (2024), no.~3, Paper No. 39, 44 pp. 
\bibitem[MNS1]{MNS} T.~Maeno, S.~Naito, and D.~Sagaki, A presentation of the torus-equivariant quantum $K$-theory ring of flag manifolds of type $A$, Part I: the defining ideal, arXiv:2302.09485. 
\bibitem[MNS2]{MNS2} T.~Maeno, S.~Naito, and D.~Sagaki, A presentation of the torus-equivariant quantum $K$-theory ring of flag manifolds of type $A$, Part II: quantum double Grothendieck polynomials, arXiv:2305.17685. 
\bibitem[O]{Orr} D.~Orr, Equivariant $K$-theory of the semi-infinite flag manifold as a nil-DAHA module, \textit{Selecta Math. (N.S.)} \textbf{29} (2023), no. 3, Paper No. 45, 26 pp. 
\bibitem[PR]{PR} H.~Pittie and A.~Ram, A Pieri--Chevalley formula in the $K$-theory of a $G/B$-bundle, \textit{Electron. Res. Announc. Amer. Math. Soc.} \textbf{5} (1999), 102--107. 
\end{thebibliography}
\end{document}